\documentclass[11pt]{scrartcl}

\usepackage{amsfonts,amsthm,amsmath,amssymb}
\usepackage{mathrsfs,mathtools}
\usepackage{enumerate}
\usepackage{hyperref}
\usepackage{esint}
\usepackage{graphicx}
\usepackage{bm}
\usepackage{commath}
\usepackage{esint}
\DeclareMathAlphabet{\mathpzc}{OT1}{pzc}{m}{it}
\usepackage{color}

\newcommand{\abssec}[1]{\noindent\normalsize {\bfseries #1\quad }\ignorespaces}
\renewenvironment{abstract}{\abssec{Abstract}}{\par\vspace{.1in}}
\newenvironment{keywords}{\abssec{Key Words}}{\par\vspace{.1in}}
\newenvironment{AMSMOS}{\abssec{AMS subject
		classification}}{\par\vspace{.1in}}

\newtheorem{theorem}{Theorem}[section]
\newtheorem{proposition}[theorem]{Proposition}
\newtheorem{lemma}[theorem]{Lemma}
\newtheorem{corollary}[theorem]{Corollary}
\newtheorem{definition}[theorem]{Definition}
\newtheorem{example}[theorem]{Example}
\newtheorem{remark}[theorem]{Remark}
\newtheorem{assumption}[theorem]{Assumption}

\def\N{{\mathbb N}}

\def\cN{\mathcal{N}}

\newcommand{\Betrag}[1]{\left|#1\right|}
\newcommand{\klg}[1]{\{#1\}}

\def\cN{{\mathcal N}}

\def\cV{{\mathcal V}}

\def\RR{\mathbb{R}}

\def\Om{\Omega}

\def\bOm{\overline{\Om}}

\def\pOm{\partial\Omega}
\def\sgn{\operatorname{sgn}}

\numberwithin{equation}{section}

\begin{document}

\title{A note on semilinear fractional elliptic equation: analysis and discretization\thanks{The work of the first and second author is partially supported by  NSF grant DMS-1521590.  The work of the third author is partially supported by the Air Force Office of Scientific Research under the Award No: FA9550-15-1-0027}}

\author{Harbir Antil \thanks{Department of Mathematical Sciences, George Mason University, Fairfax, VA 22030, USA, \texttt{hantil@gmu.edu}}
\and Johannes Pfefferer \thanks{Chair of Optimal Control, Center of Mathematical Sciences, Technical University of Munich, Boltzmannstra{\ss}e 3, 85748 Garching by Munich, Germany, \texttt{pfefferer@ma.tum.de}}
\and Mahamadi Warma \thanks{University of Puerto Rico  (Rio Piedras Campus), College of Natural Sciences,
Department of Mathematics, PO Box 70377 San Juan PR
00936-8377 (USA), \texttt{mahamadi.warma1@upr.edu, mjwarma@gmail.com}}}

\date{\today}

\maketitle

\begin{abstract}
	In this paper we study existence, regularity, and approximation of solution to a fractional semilinear elliptic equation of order $s \in (0,1)$. We identify minimal conditions on the nonlinear term and the source which leads to existence of weak solutions and uniform $L^\infty$-bound on the solutions. Next we realize the fractional Laplacian as a Dirichlet-to-Neumann map via the Caffarelli-Silvestre extension. We introduce a first-degree tensor product finite elements space to approximate the truncated problem. We derive a priori error estimates and conclude with an illustrative numerical example.
\end{abstract}

\begin{AMSMOS}
	35S15, 26A33, 65R20, 65N12, 65N30
\end{AMSMOS}

\begin{keywords}
Spectral fractional Laplace operator, semi-linear elliptic problems, regularity of weak solutions, discretization, error estimates.
\end{keywords}

\section{Introduction}

Let $\Omega\subset\RR^N$ be a bounded open set with boundary $\pOm$. In this paper we investigate the existence, regularity, and finite element approximation of weak solutions of the following  semilinear Dirichlet problem

\begin{equation}\label{ellip-pro}
\begin{cases}
(-\Delta_D)^su +f(x,u)=g\;\;&\mbox{ in }\;\Omega,\\
u=0&\mbox{ on }\;\pOm.
\end{cases}
\end{equation} 
Here, $g$ is a given measurable function on $\Omega$, $f:\;\Omega\times\RR\to\RR$ is measurable and satisfies certain conditions (that we shall specify later), $0<s<1$ and $(-\Delta_D)^s$ denotes the spectral fractional Laplace operator, that is, the fractional $s$ power of the realization in $L^2(\Om)$ of the Laplace operator with zero Dirichlet boundary condition on $\pOm$. 

Notice that $(-\Delta_D)^s$ is a nonlocal operator and $f$ is nonlinear with respect to $u$.  This makes it  challenging to identify the minimum assumptions on $\Omega$, $f$ and $g$ in the study of the existence, uniqueness, regularity and the numerical analysis of the system. The later is the main objective of our paper. 

When $f$ is linear in $u$, problems of type \eqref{ellip-pro} have received a great deal of attention. See \cite{Caf3} for results in $\RR^N$ and \cite{NALD,CaSt,CDDS:11,ST:10} for results in bounded domains. However, \cite{Caf3,CaSt,ST:10} only deal with the linear problems, on the other hand \cite{NALD,CDDS:11} deal with a different class of semilinear problems and assumes $\Omega$ and $f$ to be smooth. We refer to \cite{NOS} where a numerical scheme to approximate the linear problem was first established. To the best of our knowledge our paper is the first work addressing the existence, regularity, and numerical approximation of \eqref{ellip-pro} with almost minimum conditions on $\Omega$, $f$ and $g$. 

We use \emph{Musielak-Orlicz spaces}, endowed with Luxemburg norm, to deal with the nonlinearity. Using the Browder-Minty theorem, we first show the existence and uniqueness of a weak solution. Additional integrability condition on $g$ brings the solution in $L^\infty(\Omega)$. For the latter result, we apply a well-known technique due to Stampacchia. However, when $\Omega$ has a Lipschitz continuous boundary and $f$ is locally Lipschitz continuous we illustrate the regularity shift. For completeness we also derive the H\"older regularity of solution for smooth $\Omega$. 

Numerical realization of nonlocal operators poses various challenges for instance, direct discretization of \eqref{ellip-pro}, by using finite elements, requires access to eigenvalues and eigenvectors of $(-\Delta_D)$ which is an intractable problem in general domains. Instead we use the so-called Caffarelli-Silvestre extension to realize the fractional power $(-\Delta_D)^s$. Such an approach is a more suitable choice for numerical methods, see \cite{NOS} for the linear case. The extension idea was introduced by Caffarelli and Silvestre in $\RR^N$ \cite{Caf3} 
and its extensions to bounded domains is given in e.g. \cite{CDDS:11,ST:10}. The extension says that fractional powers $(-\Delta_D)^s$ of the spatial operator $-\Delta_D$ can be realized as an operator that maps a Dirichlet boundary condition to a Neumann condition via an extension problem on the semi-infinite cylinder $\mathcal{C} = \Omega \times (0,\infty)$, that is, a Dirichlet-to-Neumann operator. See Section~\ref{s:CS} for more details. 

We derive a priori finite element error estimates for our numerical scheme. Our proof requires the solution to a discrete linearized problem to be uniformly bounded in $L^\infty(\Omega)$, which can be readily derived by using the inverse estimates and under the assumption $s>(N-2)/2$. As a result, when $N \ge 3$, we only have error estimates in case $s>(N-2)/2$. We notice that no restriction on $s$ is needed when $N \le 2$. In summary we are only limited by the $L^\infty(\Omega)$ regularity of the solution to a discrete linearized problem when $N \ge 3$.

Recently, fractional order PDEs have made a remarkable appearance in various scientific disciplines, and have received a great deal of attention. For instance, image processing \cite{GH:14}; nonlocal electrostatics \cite{ICH}; biophysics \cite{bio}; chaotic dynamical systems \cite{MR1604710}; finance \cite{MR2064019}; mechanics \cite{atanackovic2014fractional}, where they are used to model viscoelastic behavior \cite{MR2035411}, turbulence \cite{wow,NEGRETE} and the hereditary properties of materials \cite{MR1926470}; 
 diffusion processes \cite{Abe2005403,PSSB:PSSB2221330150}, in particular processes in disordered media, where the disorder may change the laws of Brownian motion and thus leads to anomalous diffusion \cite{MR1736459,MR1081295} and many others \cite{BV,MR2025566}. In view of the fact that most of the underlying physics in the aforementioned applications can be described by nonlinear PDEs, it is natural to analyze a prototypical semilinear PDE given in \eqref{ellip-pro}. 

The paper is organized as follows: In Section~\ref{s:spec} we provide definitions of the fractional order Sobolev spaces and the spectral Dirichlet Laplacian. These results are well known. 
Section~\ref{s:orl} is devoted to essential properties of Orlicz spaces. We also specify assumptions on $f$ and state several embedding results which are due to Sobolev embedding theorems. Our main results begin in Section~\ref{sec:weaksol}, where we first show existence and uniqueness of a weak solution $u$ to the system \eqref{ellip-pro} in Proposition~\ref{prop-exis} and later with additional integrability assumption on $g$ we obtain uniform $L^\infty$-bound on $u$ in Theorem~\ref{theo-bound}. When $\Omega$ is smooth we derive the H\"older regularity of $u$ in Corollary~\ref{cor-212}. In case $\Omega$ has a Lipschitz continuous boundary and $f$ is locally Lipschitz continuous we deduce regularity shift on $u$ in Corollary~\ref{ellip-regula}. We state the extension problem in Section~\ref{s:CS} and show the existence and uniqueness of a solution $\mathcal{U}$ to the extension problem on $\mathcal{C}:=\Omega\times(0,\infty)$ in Lemma~\ref{lem:CS}. We notice that $u = \mathcal{U}(\cdot,0)$. In Section~\ref{s:disc} we  begin the numerical analysis of our problem. We first derive the energy norm and the $L^2$-norm a priori error estimates for an intermediate linear problem in Lemma~\ref{lemma:linear}. This is followed by a uniform $L^\infty$-bound on the discrete solution to an intermediate linear problem in Lemma~\ref{lemma:Linftybounded}. We conclude with the error estimates for our numerical scheme to solve \eqref{ellip-pro} in Theorem~\ref{theorem:semilinear} and a numerical example.

\section{Analysis of the semilinear elliptic problem}\label{preli}

Throughout this section without any mention, $\Omega\subset\RR^N$ denotes an arbitrary bounded open set with boundary $\pOm$.  For each result, if a regularity of $\Omega$ is needed, then we shall specify and if no specification is given, then we mean that the result holds without any regularity assumption on the open set.

\subsection{The spectral fractional Laplacian}\label{s:spec}

Let $H_0^{1}(\Omega)=\overline{\mathcal D(\Omega)}^{H^{1}(\Omega)}$ where
\begin{align*}
H^{1}(\Omega)=\{u\in L^2(\Omega):\;\int_{\Omega}|\nabla u|^2\;dx<\infty\}
\end{align*}
is the first order Sobolev space endowed with the norm
\begin{align*}
\|u\|_{H^{1}(\Omega)}=\left(\int_{\Omega}|u|^2\;dx+\int_{\Omega}|\nabla u|^2\;dx\right)^{\frac 12}.
\end{align*}
Let $-\Delta_D$ be the realization on $L^2(\Omega)$ of the Laplace operator with the Dirichlet boundary condition. That is, $-\Delta_D$ is the positive and self-adjoint operator on $L^2(\Omega)$ associated with the closed, bilinear symmetric form
\begin{align*}
\mathcal A_D(u,v)=\int_{\Omega}\nabla u\cdot\nabla v\;dx,\;\;u,v\in H_0^{1}(\Omega),
\end{align*}
in the sense that
\begin{equation*}
\begin{cases}
D(\Delta_D)=\{u\in W_0^{1,2}(\Omega):\;\exists\;w\in L^2(\Omega),\; \mathcal A_D(u,v)=(w,v)_{L^2(\Omega)}\;\forall\;v\in H_0^{1}(\Omega)\},\\
-\Delta_Du=w.
\end{cases}
\end{equation*}
For instance if $\Omega$ has a smooth boundary, then $D(\Delta_D)=H^{2}(\Omega)\cap H_0^{1}(\Omega)$, where
\begin{align*}
H^2(\Omega):=\{u\in H^1(\Omega),\partial_{x_j}u\in H^1(\Omega),\; j=1,2,\ldots,N\}.
\end{align*}
It is well-known that $-\Delta_D$ has a compact resolvent and it eigenvalues form a non-decreasing sequence $0<\lambda_1\le\lambda_2\le\cdots\le\lambda_n\le\cdots$ satisfying $\lim_{n\to\infty}\lambda_n=\infty$. We denote by $\varphi_n$ the  orthonormal eigenfunctions associated with $\lambda_n$.

Next, for $0<s<1$,  we define the fractional order Sobolev space
\begin{align*}
H^s(\Omega):=\left\{u\in L^2(\Omega):\; \int_{\Omega}\int_{\Omega}\frac{|u(x)-u(y)|^2}{|x-y|^{N+2s}}\;dxdy<\infty\right\},
\end{align*}
and we endow it with the norm defined by
\begin{align*}
\|u\|_{H^s(\Omega)}=\left(\int_{\Omega}|u|^2\;dx+\int_{\Omega}\int_{\Omega}\frac{|u(x)-u(y)|^2}{|x-y|^{N+2s}}\;dxdy\right)^{\frac 12}.
\end{align*}
We also let 
\begin{align*}
H_0^s(\Omega):=\overline{\mathcal D(\Omega)}^{H^s(\Omega)},
\end{align*}
and
\begin{align*}
H_{00}^{\frac 12}(\Omega):=\left\{u\in L^2(\Omega):\;\int_{\Omega}\frac{u^2(x)}{\mbox{dist}(x,\pOm)}\;dx<\infty\right\}.
\end{align*}
Note that
\begin{align}\label{norm-sob-es}
\|u\|_{H_0^s(\Omega)}=\left(\int_{\Omega}\int_{\Omega}\frac{|u(x)-u(y)|^2}{|x-y|^{N+2s}}\;dxdy\right)^{\frac 12}
\end{align}
defines a norm on $H_0^s(\Omega)$.

Since $\Omega$ is assumed to be bounded we have the following continuous embedding:
\begin{equation}\label{inj1}
H_0^s(\Omega)\hookrightarrow
\begin{cases}
L^{\frac{2N}{N-2s}}(\Omega)\;\;&\mbox{ if }\; N>2s,\\
L^p(\Omega),\;\;p\in[1,\infty)\;\;&\mbox{ if }\; N=2s,\\
C^{0,s-\frac{N}{2}}(\bOm)\;\;&\mbox{ if }\; N<2s.
\end{cases}
\end{equation}

We notice that if $N\ge 2$, then $N\ge 2>2s$ for every $0<s<1$, or if $N=1$ and $0<s<\frac 12$, then $N=1>2s$, and thus the first embedding in \eqref{inj1} will be used.  If $N=1$ and $s=\frac 12$, then we will use the second embedding. Finally, if $N=1$ and $\frac 12<s<1$, then $N=1<2s$ and hence, the last embedding will be used. 

For any $s\geq0$, we also introduce the following fractional order Sobolev space
\begin{align*}
\mathbb H^s(\Omega):=\left\{u=\sum_{n=1}^\infty u_n\varphi_n\in L^2(\Omega):\;\;\|u\|_{\mathbb H^s(\Omega)}^2:=\sum_{n=1}^\infty \lambda_n^su_n^2<\infty\right\},
\end{align*}
where we recall that $\lambda_n$ are the eigenvalues of $-\Delta_D$ with associated normalized eigenfunctions $\varphi_n$ and
\begin{align*} 
u_n=(u,\varphi_n)_{L^2(\Omega)}=\int_{\Omega}u\varphi_n\;dx.
\end{align*}

It is well-known that
\begin{equation}\label{inf}
\mathbb H^s(\Omega)=
\begin{cases}
H^s(\Omega)=H_0^s(\Omega)\;\;\;&\mbox{ if }\; 0<s<\frac 12,\\
H_{00}^{\frac 12}(\Omega)\;\;&\mbox{ if }\; s=\frac 12,\\
H_0^s(\Omega)\;\;&\mbox{ if }\; \frac 12<s<1.
\end{cases}
\end{equation}
It follows from \eqref{inf} that the embedding \eqref{inj1} holds with $H_0^s(\Omega)$ replaced by $\mathbb H^s(\Omega)$.

\begin{definition}
The spectral fractional Laplacian is defined on the space $\mathbb H^s(\Omega)$ by
\begin{align*}
(-\Delta_D)^su=\sum_{n=1}^\infty\lambda_n^su_n\varphi_n\qquad \text{ with  } u_n=\int_{\Omega}u\varphi_n.
\end{align*}
\end{definition}

We notice that in this case we have
\begin{align}\label{norm-2}
\|u\|_{\mathbb H^s(\Omega)}=\|(-\Delta_D)^{\frac s2}u\|_{L^2(\Omega)}.
\end{align}
Let $\mathcal D(\Omega)$ be the space of test functions on $\Omega$, that is, the space of infinitely continuously differentiable functions with compact support in $\Omega$. Then $\mathcal D(\Omega)\hookrightarrow \mathbb H^s(\Omega)\hookrightarrow L^2(\Omega)$, so, the operator $(-\Delta_D)^s$ is unbounded, densely defined and with bounded inverse $(-\Delta_D)^{-s}$ in $L^2(\Omega)$. The following integral representation of the operator $(-\Delta_D)^s$ given in \cite[p.2 Formula (3)]{NALD} will be useful. For a.e. $x\in\Omega$,
\begin{align}\label{int-rep}
(-\Delta_D)^su(x)=\mbox{P.V.}\int_{\Omega}\left[u(x)-u(y)\right]J(x,y)\;dy+\kappa(x)u(x),
\end{align}
where, letting $K_\Omega(t,x,y)$ denote the heat kernel of the semigroup generated by the operator $-\Delta_D$ on $L^2(\Omega)$,
\begin{align*}
J(x,y)=\frac{s}{\Gamma(1-s)}\int_0^\infty\frac{K_\Omega(t,x,y)}{t^{1+s}}\;dt
\end{align*}
and 
\begin{align*}
\kappa(x)=\frac{s}{\Gamma(1-s)}\int_0^\infty\left(1-\int_{\Omega}K_\Omega(t,x,y)\;dy\right)\frac{dt}{t^{1+s}},
\end{align*}
where $\Gamma$ denotes the usual Gamma function. We mention that it follows from the properties of the kernel $K_\Omega$ that $J$ is symmetric and nonnegative; i.e. $J(x,y)=J(x,y)\ge 0$ for a.e. $x,y\in\Omega$. In addition we have that $\kappa(x)\ge 0$ for a.e. $x\in\Omega$.

For more details on these topics we refer the reader to \cite{NALD,NPV,Gru2,LM,NOS} and their references.

\subsection{Some results on Orlicz spaces}\label{s:orl}

Here we give some important properties of Orlicz type spaces that will be used throughout the paper.

\begin{assumption}\label{assum2}
  For a function
  $f:\Om\times\RR\to\RR$ we consider the following assumption:
  \begin{equation*}
     \begin{cases}
        f(x,\cdot) \text{ is odd, strictly increasing}&\text{ for a.e. } x\in\Omega,\\
        f(x,0)=0 &\text{ for a.e. } x\in \Omega,\\
        f(x,\cdot) \text{ is continuous }\;&\text{ for a.e. } x\in \Omega,\\
        f(\cdot,t)  \text{ is measurable }&\mbox{ for all } t\in\RR,\\
        \lim_{t\to\infty}f(x,t)=\infty &\text{ for a.e. } x\in \Omega.
      \end{cases}
  \end{equation*}
 \end{assumption}
 
Since $f(x,\cdot)$ is strictly increasing for a.e. $x\in\Omega$,
  it has an inverse which we denote by $\widetilde{f}(x,\cdot)$. Let
  $F,\widetilde{F}:\;\Omega\times\RR\to[0,\infty)$ be defined for a.e. $x\in\Omega$ by
  \begin{align*}
 F(x,t):=\int_0^{|t|}f(x,\tau)\;d\tau\;\mbox{ and  }\;
  \widetilde{F}(x,t):=\int_0^{|t|}\widetilde{f}(x,\tau)\;d\tau.
  \end{align*}
  The functions $F$ and $\widetilde F$ are complementary Musielak-Orlicz functions
  such that $F(x,\cdot)$ and $\widetilde F(x,\cdot)$ are complementary $\cN$-functions
  for a.e. $x\in\Omega$ (in the sense of \cite[p.229]{Adam}). 

\begin{assumption}\label{assum3}
  Under the setting of Assumption \ref{assum2}, and 
  for a.e. $x\in\Omega$, let both $F(x,\cdot)$ and $\widetilde F(x,\cdot)$ satisfy
  the global $(\triangle_2)$-condition,  
  that is, there exist two constants $c_1,c_2\in (0,1]$ independent of $x$, such that for a.e. $x\in\Omega$ and for all $t\ge 0$,
  \begin{align}\label{delta-2}
  c_1tf(x,t)\le F(x,t)\le tf(x,t)\;\mbox{ and }\; c_2t\widetilde f(x,t)\le \widetilde F(x,t)\le t\widetilde f(x,t).
  \end{align}
  \end{assumption}

 First we notice that since the functions $f,\widetilde f$ are odd and $F,\widetilde F$ are even functions, we have that if \eqref{delta-2} holds, then it also holds for all $t\in\RR$.
Second, Assumption \ref{assum3} is equivalent to saying that the Musielak-Orlicz functions $F$ and $\widetilde F$ satisfy the $(\triangle_2^0)$-condition in the sense that there exist two constants $C_1,C_2>0$ such that
  \begin{align*}
  F(x,2t)\le C_1 F(x,t)\;\mbox{ and }\; \widetilde F(x,2t)\le C_2\widetilde F(x,t),\;\forall\;t\in\RR\;\mbox{ and a.e. } x\in\Omega.
  \end{align*}
  This can be easily verified by following the argument given in the monograph \cite[p.232]{Adam}.
  In that case, we let
   \begin{align*}
   L_F(\Om):=\{u:\Omega\to\RR\text{ measurable}: F(\cdot,u(\cdot))\in L^1(\Omega)\}
  \end{align*}
  be the Musielak-Orlicz space.
  The space $L_{\widetilde F}(\Om)$ is defined similarly with $F$ replaced by~$\widetilde F$.

\begin{remark}
{\em
  If Assumption \ref{assum3} holds, then by
  \cite[Theorems 1 and 2]{Doman} (see also \cite[Theorem 8.19]{Adam}),
  $L_F(\Omega)$ endowed with the Luxemburg norm given by
   \begin{align*}
\norm{u}_{F,\Omega}:=\inf\left\{k>0:\;\int_{\Omega}F \left(x,\frac{u(x)}{k}\right)\;dx\le 1\right\},
 \end{align*}
  is a reflexive Banach space. The same result also holds for $L_{\widetilde F}(\Omega)$.
  Moreover, we have the following improved H\"older inequality for Musielak-Orlicz spaces (see e.g. \cite[Formula (8.11) p.234]{Adam}):
  \begin{equation}\label{hold}
    \Betrag{\int_{\Omega}uv\;dx}\le 2\norm{u}_{F,\Omega}\norm{v}_{\widetilde{F},\Omega},\;\forall\;u\in L_F(\Omega),\; v\in L_{\widetilde F}(\Omega).
  \end{equation}
 In addition, by \cite[Corollary 5.10]{Bie}, we have that
 \begin{align}\label{coer-est}
 \lim_{\|u\|_{F,\Omega}\to\infty}\frac{\int_{\Omega}F(x,u)\;dx}{\|u\|_{F,\Omega}}=\infty.
 \end{align}
 }
\end{remark}

We have the following result.

\begin{lemma}\label{lem:hoelder}
Let Assumption \ref{assum3} hold. Then  $f(\cdot,u(\cdot))\in L_{\widetilde F}(\Om)$ for all $u\in L_F(\Omega)$.
\end{lemma}

\begin{proof}
Assume that Assumption \ref{assum3} holds. It follows from the assumptions that there exists a constant $C>0$ such that for all $\xi\in\RR$ and a.e. $x\in\Omega$,
 \begin{align*}
 \widetilde F(x,f(x,\xi)) \le \xi f(x,\xi) \le C F(x,\xi).
 \end{align*}
   Hence,
  \begin{align*}
 \int_{\Om} \widetilde F(x,f(x,u(x)))\;dx \le C\int_{\Om} F(x,u(x))\;dx <\infty
  \end{align*}
 and the proof is finished.
\end{proof}

\begin{definition}
Let $0<s<1$.  Under Assumption \ref{assum3} we can define the Banach space
  $\cV$ by
\begin{align*}
\cV:=\cV(\Omega,F):=\klg{u\in \mathbb H^{s}(\Om): F(\cdot,u(\cdot))\in L^1(\Om)}
\end{align*}
  and we endow it with the norm defined by
 \begin{align*}
\norm{u}_{\cV}:=\norm{u}_{\mathbb H^{s}(\Om)}+\norm{u}_{F,\Om}.
 \end{align*}
\end{definition}

  In this case $\cV$ is a reflexive Banach space which is continuously embedded into $\mathbb H^{s}(\Om)$. In addition, it follows from \eqref{inj1} that we have the following continuous embedding 
  \begin{align}\label{sobo1}
  \cV\hookrightarrow \mathbb H^{s}(\Om)\hookrightarrow L^{2^\star}(\Omega),
  \end{align}
where we recall that
\begin{align*}
2^\star=\frac{2N}{N-2s}\;\mbox{ if }\; N\ge 2>2s\;\mbox{ or if }\,N=1\;\mbox{ and }\; 0<s<\frac 12.
\end{align*}
 If $N=1$ and $s=\frac 12$, then $2^{\star}$ is any number in the interval $[1,\infty)$. If $N=1$ and $\frac 12<s<1$, then we have the continuous embedding
\begin{align}\label{sobo2}
\cV\hookrightarrow \mathbb H^{s}(\Om)\hookrightarrow C^{0,s-\frac 12}(\bOm).
\end{align}

\subsection{Weak solutions of the semilinear problem}\label{sec:weaksol}

Now we can introduce our notion of weak solutions to the system \eqref{ellip-pro}.

We recall that we have set $\cV:=\mathbb H^{s}(\Omega)\cap L_F(\Omega)$. We shall denote by $\cV^\star=(\mathbb H^{s}(\Omega)\cap L_F(\Omega))^\star$ the dual of the reflexive Banach space $\cV$ and by $\langle\cdot,\cdot\rangle$ their duality map.

\begin{definition}
A function $u\in \cV$ is said to be a weak solution of  \eqref{ellip-pro} if the identity
\begin{align}\label{form-ws}
\mathcal F(u,v):=\int_{\Omega}(-\Delta_D)^{\frac s2}u(-\Delta_D)^{\frac s2}v\;dx+\int_{\Omega}f(x,u)v\;dx=\langle g,v\rangle,
\end{align}
holds for every $v\in \cV$ and the right hand side $g\in \cV^\star$.
\end{definition}

We have the following result of existence and uniqueness of weak solution.

\begin{proposition}[Existence of weak solution]\label{prop-exis}
Let Assumption \ref{assum3} hold.  Then for every $g\in \cV^\star$, the system \eqref{ellip-pro} has a unique weak solution $u$. In addition, if $g\in \mathbb H^{-s}(\Omega):=(\mathbb H^s(\Omega))^\star$, then there exists a constant $C>0$ such that 
\begin{equation}\label{nor-est}
\|u\|_{\mathbb H^s(\Omega)}\le C\|g\|_{\mathbb H^{-s}(\Omega)}.
\end{equation}
\end{proposition}

\begin{proof}
Let $u\in \cV$ be fixed. First we notice that it follows from Lemma \ref{lem:hoelder} that $f(\cdot,u(\cdot))\in L_{\widetilde F}(\Omega)$. Next, using the classical H\"older inequality and \eqref{hold} we have that for all $v\in \cV$,
\begin{align}\label{ine-bound}
|\mathcal F(u,v)|\le& \|(-\Delta_D)^{\frac s2}u\|_{L^2(\Omega)}\|(-\Delta_D)^{\frac s2}v\|_{L^2(\Omega)}+2\|f(\cdot,u)\|_{\widetilde F,\Omega}\|v\|_{F,\Omega}\notag\\
\le &\left(\|(-\Delta_D)^{\frac s2}u\|_{L^2(\Omega)}+2\|f(\cdot,u)\|_{\widetilde F,\Omega}\right)\|v\|_{\cV}.
\end{align}
Since $\mathcal F(u,\cdot)$ is linear (in the second variable) we have shown that $\mathcal F(u,\cdot)\in \cV^\star$ for every $u\in \cV$. Since $f(x,\cdot)$ is strictly monotone, we have that every $u,v\in \cV$, $u\ne v$,
\begin{align*}
\mathcal F(u,u-v)-\mathcal F(v,u-v)>0.
\end{align*}
Hence, $\mathcal F$ is strictly monotone.  It follows from the continuity of the norm function and the continuity of $f(x,\cdot)$ that $\mathcal F$ is hemi-continuous. It follows also from the $(\Delta_2)$-condition and \eqref{coer-est} that
\begin{align*}
\lim_{\|u\|_{ F,\Omega}\to\infty}\frac{\int_{\Omega}f(x,u)u\;dx}{\|u\|_{F,\Omega}}=\infty,
\end{align*}
and this implies that
\begin{align*}
\lim_{\|u\|_{\cV}\to\infty}\frac{\mathcal F(u,u)}{\|u\|_{\cV}}=\infty.
\end{align*}
Hence, $\mathcal F$ is coercive. We have shown that for every $u\in \cV$ there exists a unique $A_F\in \cV^\star$ such that $\mathcal F(u,v)=\langle A_F(u),v\rangle$ for every $v\in \cV$. This defines an operator $A_F:\; \cV\to \cV^\star$ which is hemi-continuous, strictly monotone, coercive and bounded (the boundedness follows from \eqref{ine-bound}). Therefore $A_F(\cV)=\cV^\star$ and hence, by the Browder-Minty theorem, for every $g\in \cV^\star$, there exists a unique $u\in \cV$ such that $A_F(u)=v$. Now assume that $g\in \mathbb H^{-s}(\Omega)\hookrightarrow \cV^\star$. Then taking $v=u$ in \eqref{form-ws}, using the fact that $f(x,u)u\ge 0$ and noticing that $\langle g,u\rangle_{\cV^\star,\cV}=\langle g,u\rangle_{ \mathbb H^{-s}(\Omega), \mathbb H^{s}(\Omega)}$ (recall that $g\in\mathbb H^{-s}(\Omega)$ and $u\in\mathbb H^s(\Omega)$) we get that
\begin{align*}
\|u\|_{\mathbb H^s(\Omega)}^2\le |\langle g,u\rangle|\le \|g\|_{\mathbb H^{-s}(\Omega)}\|u\|_{\mathbb H^s(\Omega)}.
\end{align*}
We have shown \eqref{nor-est} and the proof is finished.
\end{proof}

The following theorem is the main result of this section.

\begin{theorem}\label{theo-bound}
Let Assumption \ref{assum3} hold and that $g\in L^p(\Omega)$ with 
\begin{equation}\label{cond-p}
\begin{cases}
p>\frac{N}{2s}\;\;&\mbox{ if }\; N>2s,\\
p>1 \;\;&\mbox{ if }\; N=2s,\\
p=1\;\;&\mbox{ if }\; N<2s.
\end{cases}
\end{equation}
Then every weak solution $u$ of  \eqref{ellip-pro} belongs to $L^\infty(\Omega)$. Moreover there is a constant $C=C(N,s,p,\Omega)>0$ such that
\begin{align}\label{inf-norm}
\|u\|_{L^\infty(\Omega)}\le C\|g\|_{L^p(\Omega)}.
\end{align}
\end{theorem}

\begin{remark}
We mention that if $N=1$ and $\frac 12<s<1$, then it follows from \eqref{sobo2} that the weak solution of \eqref{ellip-pro} is globally H\"older continuous on $\bOm$ and in this case there is nothing to prove. Thus we need to prove the theorem only in the cases $N\ge 2$, or $N=1$ and $0<s\le \frac 12$.
\end{remark}

To prove the theorem we need the following lemma which is of analytic nature and will be useful in deriving some a priori estimates of weak solutions of elliptic type equations (see e.g. \cite[Lemma B.1.]{kinderlehrer1980}). 

\begin{lemma}\label{lem-01}
Let $\Xi = \Xi(t)$ be a nonnegative, non-increasing function on a half line $t\ge k_0\ge 0$ such that there are positive constants $c, \alpha$ and $\delta$ ($\delta >1$) with
\begin{equation*}
\Xi(h) \le c(h-k)^{-\alpha}\Xi(k)^{\delta}\mbox{ for }  h>k\ge k_0.
\end{equation*}
Then
\begin{equation*}
\Xi(k_0+d) = 0\quad \mbox{ with }\quad d^{\alpha}= c \Xi(k_0)^{\delta -1}2^{\alpha\delta/(\delta -1)}.
\end{equation*}
\end{lemma}

\begin{proof}[Proof of Theorem \ref{theo-bound}]
Invoking Assumption \ref{assum3} and $g\in L^p(\Omega)$ with $p$ satisfying \eqref{cond-p}, it follows from \eqref{sobo1} that $g\in \cV^\star$. Hence, by Proposition \ref{prop-exis}, the system \eqref{ellip-pro} has a unique weak solution $u\in \cV$. We prove the result in two steps.

{\em Step 1}. Let $u\in \cV$, $k\ge 0$ and set $u_k:=(|u|-k)^+\sgn(u)$. Using \cite[Lemma 2.7]{War} we get that $u_k\in \cV$.
We claim that
\begin{align}\label{claim1}
\mathcal F(u_k,u_k)\le\mathcal F(u,u_k).
\end{align}
Indeed, let $A_k:=\{x\in\Om:\;|u(x)|\ge k\}$, $A_k^+:=\{x\in\Om:\;u(x)\ge k\}$ and $A_k^-:=\{x\in\Om:\;u(x)\le -k\}$ so that $A_k=A_k^+\cup A_k^-$. Then
\begin{equation}\label{for-uk}
u_k=
\begin{cases}
u-k\;\;&\mbox{ in }\;A_k^+,\\
u+k &\mbox{ in }\;A_k^-,\\
0 &\mbox{ in }\; \Om\setminus A_k.
\end{cases}
\end{equation}
Since $f(x,\cdot)$ is odd, monotone increasing and $0\le u_k=u-k\le u$ on $A_k^+$, we have that for a.e. $x\in A_k^+$,
\begin{align}\label{A1}
f(x,u_k)u_k=f(x,u-k)u_k\le  f(x,u)u_k.
\end{align}
Similarly, since $u\le u+k=u_k\le 0$ on $A_k^-$, it follows that for a.e. $x\in A_k^-$,
\begin{align}\label{A2}
f(x,u_k)u_k=f(x,u+k)u_k\le  f(x,u)u_k.
\end{align}
It follows from \eqref{A1} and \eqref{A2} that for every $k\ge 0$,
\begin{align}\label{A1-A2}
\int_{\Omega}f(x,u_k)u_k\;dx\le \int_{\Omega}f(x,u)u_k\;dx.
\end{align}
Next, we show that for every $k\ge 0$,
\begin{align}\label{pf-claim1-2}
\int_{\Omega}(-\Delta_D)^{\frac s2}u_k(-\Delta_D)^{\frac s2}u_k\;dx\le \int_{\Omega}(-\Delta_D)^{\frac s2}u(-\Delta_D)^{\frac s2}u_k\;dx.
\end{align}
We notice that it follows from the integral representation \eqref{int-rep} that
\begin{align*}
&\int_{\Omega}(-\Delta_D)^{\frac s2}u_k(-\Delta_D)^{\frac s2}u_k\;dx=\|u_k\|_{\mathbb H^s(\Omega)}^2\\
=&\frac 12\int_{\Omega}\int_{\Omega}|u_k(x)-u_k(y)|^2J(x,y)\;dxdy+\int_{\Omega}\kappa(x)|u_k(x)|^2\;dx.
\end{align*}
Calculating and using \eqref{for-uk} we get that for every $k\ge 0$,
\begin{align}\label{pf-claim1}
\int_{\Omega}\int_{\Omega}&|u_k(x)-u_k(y)|^2J(x,y)\;dxdy\\
=&\int_{A_k^+}\int_{A_k^+}(u(x)-u(y))(u_k(x)-u_k(y))J(x,y)\;dxdy\notag\\
&+\int_{A_k^+}\int_{A_k^-}|u(x)-u(y)-2k|^2J(x,y)\;dxdy\notag\\
&+\int_{A_k^-}\int_{A_k^-}(u(x)-u(y))(u_k(x)-u_k(y))J(x,y)\;dxdy\notag\\
&+\int_{A_k^-}\int_{A_k^+}|u(x)-u(y)+2k|^2J(x,y)\;dxdy\notag\\
&+\int_{\Omega\setminus A_k}\int_{A_k}|u_k(y)|^2J(x,y)\;dxdy\notag\\
&+\int_{A_k}\int_{\Omega\setminus A_k}|u_k(x)|^2J(x,y)\;dxdy\notag.
\end{align}
Since $u(x)-u(y)-2k\ge 0$ for a.e. $(x,y)\in A_k^+\times A_k^-$, we have that for a.e. $(x,y)\in A_k^+\times A_k^-$,
\begin{align}\label{EW1}
(u(x)-u(y)-2k)^2&\le (u(x)-u(y)) (u(x)-u(y)-2k)\\
&=(u(x)-u(y))(u_k(x)-u_k(y)).\notag
\end{align}
 Since $u(x)-u(y)+2k\le 0$ for a.e $(x,y)\in A_k^-\times A_k^+$, it follows that for a.e $(x,y)\in A_k^-\times A_k^+$,
\begin{align}\label{EW2} 
 (u(x)-u(y)+2k)^2&\le (u(x)-u(y)) (u(x)-u(y)+2k)\\
 &=(u(x)-u(y))(u_k(x)-u_k(y)).\notag 
\end{align} 
 For a.e. $(x,y)\in (\Omega \setminus A_k)\times A_k$, we have that (recall that $u_k(x)=0$),
\begin{align}\label{M}
(u(x)-u(y))(u_k(x)-u_k(y))=-(u(x)-u(y))u_k(y)=(u(y)-u(x))u_k(y).
\end{align}
Using \eqref{M} we get the following estimates:
\begin{itemize}
\item For a.e. $(x,y)\in (\Omega\setminus A_k)\times A_k^+$ we have that (as $k-u(x)> 0$ and $u(y)-k\ge 0$)
\begin{align}\label{M1}
(u(x)-u(y))(u_k(x)-u_k(y))=&(u(y)-k+k-u(x))(u(y)-k)\notag\\
=&(u(y)-k)^2+(k-u(x))(u(y)-k)\notag\\
\ge & (u(y)-k)^2=|u_k(y)|^2.
\end{align}
\item For a.e. $(x,y)\in (\Omega\setminus A_k)\times A_k^-$ we have that (as $k+u(x) > 0$ and $u(y)+k\le 0$)
\begin{align}\label{M2}
(u(x)-u(y))(u_k(x)-u_k(y))=&(u(y)+k-k-u(x))(u(y)+k)\notag\\
=&(u(y)+k)^2-(k+u(x))(u(y)+k)\notag\\
\ge& (u(y)+k)^2=|u_k(y)|^2.
\end{align}
\end{itemize}
Combining \eqref{M1} and \eqref{M2} yields
for a.e. $(x,y)\in (\Omega\setminus A_k)\times A_k$ 
\begin{align}\label{M12}
(u(x)-u(y))(u_k(x)-u_k(y))\ge |u_k(y)|^2.
\end{align}
Proceeding in the same manner, we also get that for a.e. $(x,y)\in A_k\times (\Omega\setminus A_k)$ (recall that here $u_k(y)=0$), 
\begin{align}\label{M3}
(u(x)-u(y))(u_k(x)-u_k(y))\ge |u_k(x)|^2.
\end{align}
Using \eqref{EW1}, \eqref{EW2}, \eqref{M12}, and \eqref{M3} we get from \eqref{pf-claim1} that for every $k\ge 0$ (recall that $J(x,y)\geq 0$ for a.e. $x,y\in \Omega$),
 \begin{align}\label{pf-claim1-0}
\int_{\Omega}\int_{\Omega}&|u_k(x)-u_k(y)|^2J(x,y)\;dxdy\\
\le &\int_{\Omega}\int_{\Omega}(u(x)-u(y))(u_k(x)-u_k(y))J(x,y)\;dxdy.\notag
\end{align}
As for \eqref{A1-A2} we have that for every $k\ge 0$ (recall that $\kappa(x)\ge 0$ for a.e. $x\in\Omega$),
\begin{align}\label{kappa}
\int_{\Omega}\kappa(x)|u_k(x)|^2\;dx\le \int_{\Omega}\kappa(x)u(x)u_k(x)\;dx.
\end{align}
Now the estimate \eqref{pf-claim1-2} follows from \eqref{pf-claim1-0} and \eqref{kappa} since according to \eqref{int-rep} there holds
\begin{align*}
\int_{\Omega}(-\Delta_D)^{\frac s2}u(-\Delta_D)^{\frac s2}u_k\;dx&=\frac 12\int_{\Omega}\int_{\Omega}(u(x)-u(y))(u_k(x)-u_k(y))J(x,y)\;dxdy\\
&\quad+\int_{\Omega}\kappa(x)u(x)u_k(x)\;dx.
\end{align*}
It follows from \eqref{A1-A2} and \eqref{pf-claim1-2} that for every $k\ge 0$,
\begin{align*}
\mathcal F(u_k,u_k)=&\int_{\Omega}(-\Delta_D)^{\frac s2}u_k(-\Delta_D)^{\frac s2}u_k\;dx+\int_{\Omega}f(x,u_k)u_k\;dx\\
\le &\int_{\Omega}(-\Delta_D)^{\frac s2}u(-\Delta_D)^{\frac s2}u_k\;dx+\int_{\Omega}f(x,u)u_k\;dx\\
\le &\mathcal F(u,u_k),
\end{align*}
and we have proved the claim \eqref{claim1}.

{\em Step 2}. Let $u\in\cV$ be the unique weak solution of the system \eqref{ellip-pro}, $k\ge 0$ and let $u_k$ be as above. Let $p_1\in [1,\infty]$ be such that $\frac{1}{p}+\frac{1}{2^\star}+\frac{1}{p_1}=1$ where we recall that $2^\star=\frac{2N}{N-2s}>2$. Since $p>\frac{N}{2s}=\frac{2^\star}{2^\star-2}$, we have that
\begin{align}\label{eq-B}
\frac{1}{p_1}=1-\frac{1}{2^\star}-\frac{1}{p}>\frac{2^\star}{2^\star}-\frac{1}{2^\star}-\frac{2^\star-2}{2^\star}=\frac{1}{2^\star}\Longrightarrow p_1<2^\star.
\end{align}
Taking $v=u_k$ as a test function in \eqref{form-ws} and using the classical H\"older inequality we get that there exists a constant $C=C(N,s,p)>0$ such that
\begin{align}\label{est}
\mathcal F(u,u_k)=\int_{\Omega}gu_k\;dx\le& \|g\|_{L^p(\Omega)}\|u_k\|_{L^{2^\star}(\Omega)}\|\chi_{A_k}\|_{L^{p_1}(\Omega)},
\end{align}
where $\chi_{A_k}$ denotes the characteristic function of the set $A_k$.
Using \eqref{claim1}, \eqref{est}, \eqref{sobo1} and the fact that $\int_{\Om}f(x,u_k)u_k\;dx\ge 0$, we get that there exist two constants $C,C_1>0$ such that for every $k\ge 0$,
\begin{align*}
C\|u_k\|_{L^{2^\star}(\Omega)}^2 &\le \|u_k\|_{\mathbb H^{s}(\Omega)}^2\le \mathcal F(u_k,u_k)\le \mathcal F(u,u_k) \\
&\le C_1 \|g\|_{L^p(\Omega)}\|u_k\|_{L^{2^\star}(\Omega)}\|\chi_{A_k}\|_{L^{p_1}(\Omega)},
\end{align*}
and this implies that there exists a constant $C>0$ such that for every $k\ge 0$,
\begin{align}\label{est2}
\|u_k\|_{L^{2^\star}(\Omega)}\le C \|g\|_{L^p(\Omega)}\|\chi_{A_k}\|_{L^{p_1}(\Omega)}.
\end{align}
Let $h>k$. Then $A_h\subset A_k$ and on $A_h$ we have that $|u_k|\ge h-k$. Therefore, it follows from \eqref{est2} that for every $h>k\ge 0$,
\begin{align}\label{B1}
\|\chi_{A_h}\|_{L^{2^\star}(\Omega)}\le C (h-k)^{-1}\|g\|_{L^p(\Omega)}\|\chi_{A_k}\|_{L^{p_1}(\Omega)}.
\end{align}
Let $\delta:=\frac{2^\star}{p_1}>1$ by \eqref{eq-B}. Then using the H\"older inequality again we get that there exists a constant $C>0$ such that for every $k\ge 0$, we have
\begin{align}\label{B2}
\|\chi_{A_k}\|_{L^{p_1}(\Omega)}\le C \|\chi_{A_k}\|_{L^{2^\star}(\Omega)}^\delta.
\end{align}
It follows from \eqref{B1} and \eqref{B2} that there exists a constant $C>0$ such that for every $h>k\ge 0$,
\begin{align*}
\|\chi_{A_h}\|_{L^{2^\star}(\Omega)}\le C (h-k)^{-1}\|g\|_{L^p(\Omega)}\|\chi_{A_k}\|_{L^{2^\star}(\Omega)}^\delta.
\end{align*}
It follows from Lemma \ref{lem-01} with $\Xi(k)=\|\chi_{A_k}\|_{L^{2^\star}(\Omega)}$ that there exists a constant $C_1>0$ such that
\begin{align*}
\|\chi_{A_K}\|_{L^{2^\star}(\Omega)}=0\;\mbox{ with }\;K=CC_1\|g\|_{L^p(\Omega)}.
\end{align*}
We have shown the estimate \eqref{inf-norm} and the proof is finished.
\end{proof}

We have the following improved regularity of weak solutions to  the system \eqref{ellip-pro}, in case $\Omega$ is a smooth open set.

\begin{corollary}[Regularity: $\Omega$ smooth]\label{cor-212} 
Let $\Omega\subset\RR^N$ be a bounded open set with smooth boundary.
Let Assumption \ref{assum3} hold and that 
\begin{align}\label{con-f}
f(\cdot,t)\in L^\infty(\Omega),\;\; \forall\;t\in\RR,\; |t|\le \alpha\;\mbox{ for some constant }\alpha>0. 
\end{align}
Then the following assertions hold.
\begin{enumerate}
\item Let $g\in L^p(\Omega)$ with $p$ as in \eqref{cond-p}. If $2s-\frac Np\ne 1$ (resp. $2s-\frac Np=1$) then the weak solution of \eqref{ellip-pro} belongs to $C^{0,2s-\frac Np}(\bOm)$ (resp. $C_{\star}^1(\bOm)$), where $C_{\star}^1(\bOm)$ is the H\"older-Zygmund space.

\item If $g\in L^\infty(\Omega)$, then $u\in \cap_{\varepsilon>0}C^{0,2s-\varepsilon}(\bOm)$.
\end{enumerate}
\end{corollary}

\begin{proof}
Let Assumption \ref{assum3} hold and that $f$ satisfies \eqref{con-f}. 
Let $g\in L^p(\Omega)$ with $p$ as in part (a) or part  (b). Then by Theorem \ref{theo-bound} the solution $u\in L^\infty(\Omega)$. Hence,  by \eqref{con-f} we have that the function $f(\cdot,u(\cdot))\in L^\infty(\Omega)$. Let then $h:=g-f(\cdot,u(\cdot))$. Then $h$ belongs to same space as the function $g$ and $u$ is a weak solution of the Dirichlet problem 
\begin{align*}
(-\Delta_D)^su=h\;\;\mbox{ in }\;\Omega,\; u=0\;\mbox{ on }\;\pOm.
\end{align*}
Now the regularity of $u$ given in part (a) and part (b) follows from \cite[Corollary 3.5]{Gru2}. 
\end{proof}

For all the results presented so far, Assumption \ref{assum3} is sufficient. However, to show higher regularity in $\mathbb H^{2s+\beta}(\Omega)$ with $0 \le \beta < 1$ and for the discretization error estimates in the sequel, we need an assumption on the local Lipschitz continuity of the nonlinearity in addition.

\begin{assumption}\label{assum-lip}
	For all $M>0$ there exists a constant $L_{M}>0$ such that $f$ satisfies 
	\begin{align*}
		|f(x,u_1)-f(y,u_2)|\leq L_{M}|u_1-u_2|
	\end{align*}
	for a.e. $x, y\in\Omega$ and $u_i\in\mathbb{R}$ with $|u_i|\leq M$, $i=1,2$. 
\end{assumption}

The following result will be frequently used throughout the paper.

\begin{lemma}\label{lem-f-lip}
Let $0\le \beta< 1$ and assume that $f$ satisfies Assumption \ref{assum-lip}. Then for every $u\in\mathbb H^\beta(\Omega)\cap L^\infty(\Omega)$, we have that $f(\cdot,u(\cdot))\in\mathbb H^\beta(\Omega)$. 
\end{lemma}

\begin{proof}
We notice that if $\beta=0$ then there is nothing to prove.
Let then $0< \beta< 1$ and $u\in \mathbb H^\beta(\Omega)\cap L^\infty(\Omega)$. Since $f(x,0)=0$,  $u\in L^2(\Omega)$, $|u(x)|\le M$ for a.e. $x\in\Omega$, for some constant $M>0$, we have that  (by Assumption \ref{assum-lip})
\begin{align}\label{eq-lip}
|f(x,u(x))|=|f(x,u(x))-f(x,0)|\le L_M|u(x)|\;\;\mbox{ for a.e. }\;x\in\Omega.
\end{align}
This implies that $f(\cdot,u(\cdot))\in L^2(\Omega)$.  Assumption \ref{assum-lip} also implies that 
\begin{align}\label{eq-lip2}
|f(x,u(x))-f(x,u(y))|\le L_M|u(x)-u(y)|\;\;\mbox{ for a.e. }\;x, y\in\Omega.
\end{align}
We have the following three cases.
\begin{itemize}
\item If $\beta=\frac 12$, then using \eqref{eq-lip} we obtain that
\begin{align*}
\int_{\Omega}\frac{|f(x,u(x))|^2}{\mbox{dist}(x,\pOm)}\;dx\le L_M^2\int_{\Omega}\frac{|u(x)|^2}{\mbox{dist}(x,\pOm)}\;dx<\infty.
\end{align*}
Hence, $f(\cdot,u(\cdot))\in\mathbb H^{\frac 12}(\Omega)$.

\item If $0<\beta<\frac 12$, then it follows from \eqref{eq-lip2} that
\begin{align}\label{eq-Int}
\int_{\Omega}\int_{\Omega}\frac{|f(x,u(x))-f(y,u(y))|^2}{|x-y|^{N+2\beta}}\;dxdy\le L_M^2\int_{\Omega}\int_{\Omega}\frac{|u(x)-u(y)|^2}{|x-y|^{N+2\beta}}\;dxdy<\infty,
\end{align}
and this implies that $f(\cdot,u(\cdot))\in H^\beta(\Omega)=\mathbb H^\beta(\Omega)$.

\item If $\frac 12<\beta<1$, then the estimate \eqref{eq-Int} also holds and this implies that $f(\cdot,u(\cdot))\in H^\beta(\Omega)$. Since $f(x,0)=0$ for a.e. $x\in\Omega$, we also get that $f(\cdot,u(\cdot))\in \mathbb H^\beta(\Omega)$ by approximation if necessary.
\end{itemize}
The proof of the lemma is finished.
\end{proof}

We have the following elliptic regularity.

\begin{corollary}[Regularity: $\Omega$ Lipschitz]\label{ellip-regula}
Let $\Omega\subset\RR^N$ be a bounded open set with Lipschitz continuous boundary. Assume Assumptions \ref{assum3} and \ref{assum-lip} are fulfilled. In addition, let $0\le \beta< 1$, $g\in \mathbb H^\beta(\Omega)\cap L^p(\Omega)$ with $p$ as in \eqref{cond-p}and let $u\in\mathbb H^s(\Omega)$ be the weak solution of \eqref{ellip-pro}. Then $u\in \mathbb H^{2s+\beta}(\Omega)$.
\end{corollary}

\begin{proof}
In view of the assumption on $f$ and $g$, it follows from Proposition \ref{prop-exis} and Theorem \ref{theo-bound} that the system \eqref{ellip-pro} has a unique weak solution $u\in\mathbb H^s(\Omega)\cap L^\infty(\Omega)$. Since  $f(\cdot,u(\cdot))\in L^2(\Omega)$ (by Lemma \ref{lem-f-lip}) we have that $g-f(\cdot,u(\cdot))\in L^2(\Omega)$ then
\begin{equation}\label{eq:un}
	u_n=\lambda_n^{-s}\int_{\Omega}(g-f(\cdot,u))\varphi_n,\quad n\in\N.
\end{equation}
Using the $\mathbb{H}^{2s}$ norm definition we arrive at
\[
 \|u\|_{\mathbb{H}^{2s}(\Omega)}^2 = \| g - f(\cdot,u) \|_{L^2(\Omega)}^2 ,
\]
i.e., $u\in \mathbb H^{2s}(\Omega)\cap L^\infty(\Omega)$ (see also \cite[Section 2 pp.772-773]{CaSt}). We have two cases.
\begin{itemize}
\item If $2s\ge 1$, then $u\in\mathbb H^\beta(\Omega)$  (recall that $0<\beta< 1$) and hence, $f(\cdot,u(\cdot))\in \mathbb H^\beta(\Om)$ by Lemma \ref{lem-f-lip}. We have shown that  $g-f(\cdot,u(\cdot))\in \mathbb H^{\beta}(\Omega)$. Since $g-f(\cdot,u(\cdot))\in \mathbb{H}^\beta(\Omega)$, using \eqref{eq:un} and the definition of $\mathbb{H}^{2s+\beta}$ we obtain
\begin{align*}
	\|u\|_{\mathbb{H}^{2s+\beta}(\Omega)}^2&=\sum_{n=1}^\infty u_n^2\lambda_n^{2s+\beta}=\sum_{n=1}^\infty \left(\lambda_n^{-s}\int_{\Omega}(g-f(\cdot,u))\varphi_n\right)^2\lambda_n^{2s+\beta}\\
	&=\|g-f(\cdot,u)\|_{\mathbb{H}^{\beta}(\Omega)}^2,
\end{align*}
and we have shown that $u\in \mathbb H^{2s+\beta}(\Omega)$ (see also e.g. \cite[Section 2]{CaSt}).

\item If $2s<1$, then $f(\cdot,u(\cdot))\in \mathbb H^{2s}(\Omega)$ (by Lemma \ref{lem-f-lip}) and this implies that $g-f(\cdot,u(\cdot))\in \mathbb H^{\min\{2s,\beta\}}(\Omega)$. As above we then get that $u\in \mathbb H^{2s+\min\{2s,\beta\}}(\Omega)$. Repeating the  same argument with $2s+\min\{2s,\beta\}$ in place of $2s$ and so on, we can arrive that in fact $g-f(\cdot,u(\cdot))\in \mathbb H^{\beta}(\Omega)$ and as above this implies that $u\in \mathbb H^{2s+\beta}(\Omega)$.
\end{itemize}
The proof is finished. 
\end{proof}

We conclude this section with the following example.

\begin{example}
Let $q\in [1,\infty)$ and let $b: \Omega\to (0,\infty)$ be a  function in $L^\infty(\Omega)$, that is, $b(x)> 0$ for a.e. $x\in\Omega$. Define the function $f:\Omega\times\mathbb R\to\mathbb R$ by $f(x,t)=b(x)|t|^{q-1}t$. It is clear that $f$ satisfies Assumption \ref{assum2} and the associated function $F:\Omega\times\mathbb R\to [0,\infty)$ is given by $F(x,t)=\frac{1}{q+1}b(x)|t|^{q+1}$. For a.e. $x\in\Omega$, the inverse $\widetilde f(x,\cdot)$ of $f(x,\cdot)$ is given by $\widetilde f(x,t)=\left(b(x)\right)^{-\frac 1q}|t|^{\frac{1-q}{q}}t$. Therefore, the complementary function $\widetilde F$ of $F$ is given by $\widetilde F(x,t)=\frac{q}{q+1}\left(b(x)\right)^{-\frac 1q}|t|^{\frac{q+1}{q}}$. Hence,
\begin{align*}
tf(x,t)=(q+1) F(x,t)\;\mbox{ and }\; t\widetilde f(x,t)=\frac{q+1}{q}\widetilde F(x,t),
\end{align*}
and we have shown that Assumption \ref{assum3} is also satisfied. Moreover, we have that $f$ satisfies \eqref{con-f} in Corollary \ref{cor-212}.
In particular, if $b(x)=C$ for a.e. $x\in\Omega$, for some constant $C>0$, then the function $f$ also satisfies Assumption \ref{assum-lip}.
\end{example}

\section{The extended problem in the sense of Caffarelli and Silvestre}\label{s:CS}

In case that the nonlinearity $f(x,t)$ is identically zero, it is well known that 
problem \eqref{ellip-pro} can equivalently be posed on a semi-infinite cylinder. This approach is originally due to Caffarelli and Silvestre \cite{Caf3}. While they assume the unbounded domain $\RR^N$, the restriction to bounded domains was considered in \cite{CT:10,CDDS:11,ST:10}. We mention that for the existence and uniqueness of solutions to the problem on this semi-infinite cylinder it is sufficient to consider an open set with a Lipschitz continuous boundary, see \cite[Theorem~2.5]{CaSt} for details.
We operate under the same setup in the present section.
Since we will send the non-linearity in \eqref{ellip-pro} to its right hand side, it is straightforward to introduce the extended problem in the semi-linear case.

We begin by introducing the required notation. In the following, we denote by $\mathcal{C}$ the aforementioned semi-infinite cylinder with base $\Omega$, i.e., $\mathcal{C}=\Omega\times(0,\infty)$, and its lateral boundary by $\partial_L\mathcal{C}:=\partial\Omega\times[0,\infty)$. For later purposes, we also introduce for any $\mathpzc{Y}>0$ a truncation of the cylinder $\mathcal{C}$ by $\mathcal{C}_\mathpzc{Y}:=\Omega\times (0,\mathpzc{Y})$. Similar to the lateral boundary $\partial_L\mathcal{C}_\mathpzc{Y}$, we set $\partial_L\mathcal{C}_\mathpzc{Y}:=\partial\Omega\times[0,\mathpzc{Y}]$. Consequently, the semi-infinite cylinder and its truncated version are objects defined in $\mathbb{R}^{N+1}$. Throughout the remaining part of the paper, $y$ denotes the extended variable, such that a vector $x'\in\mathbb{R}^{N+1}$ admits the representation $x'=(x_1,\ldots,x_N,x_{N+1})=(x,x_{N+1})=(x,y)$ with $x_i\in \mathbb{R}$ for $i=1,\ldots,N+1$, $x\in\mathbb{R}^N$ and $y\in \mathbb{R}$.

Due to the degenerate/singular nature of the extended problem by Caffarelli and Silvestre, it will be necessary to discuss the solvability of this problem in certain weighted Sobolev spaces with weight function $y^\alpha$, $\alpha\in(-1,1)$, see \cite[Section~2.1]{Turesson}, \cite{KO84} and \cite[Theorem~1]{GU} for a more sophisticated discussion of such spaces. In this regard, let $\mathcal{D}\subset \mathbb{R}^{N}\times[0,\infty)$ be an open set, such as $\mathcal{C}$ or $\mathcal{C}_\mathpzc{Y}$, then we define the weighted space $L^2(y^\alpha,\mathcal{D})$ as the space of all measurable functions defined on $\mathcal{D}$ with finite norm $\|w\|_{L^2(y^\alpha,\mathcal{D})}:=\|y^{\alpha/2}w\|_{L^2(\mathcal{D})}$. Similarly, using a standard multi-index notation, the space $H^1(y^\alpha,\mathcal{D})$ denotes the space of all measurable functions $w$ on $\mathcal{D}$ whose weak derivatives $D^\delta w$ exist for $|\delta|=1$ and fulfill
\[
	\|w\|_{H^1(y^\alpha,\mathcal{D})}:=\left(\sum_{|\delta|\leq 1}\|D^\delta w\|^2_{L^2(y^\alpha,\mathcal{D})}\right)^{1/2}<\infty.
\]
To study the extended problems we also need to introduce the space
\[
  \mathring{H}^1_L(y^\alpha,\mathcal{C}):=\{w\in H^1(y^\alpha,\mathcal{C}):w=0\text{ on } \partial_L\mathcal{C}\}.
\]
The space $\mathring{H}^1_L(y^\alpha,\mathcal{C}_\mathpzc{Y})$ is defined in an analogous manner.
Formally, we need to indicate the trace of a function on $\Omega$ by introducing the trace mapping on $\Omega$. However, we skip this notation since it will be clear whenever we speak about traces.

Now, the extended problem reads as follows: Given $g\in \mathcal{V}^\star$, find $\mathcal{U}\in \mathring{H}^1_L(y^\alpha,\mathcal{C})$ such that 
\begin{equation}\label{eq:extendedweak}
\int_\mathcal{C} y^\alpha \nabla \mathcal{U}\cdot\nabla \Phi\;dxdy+d_s\int_{\Omega}f(x,\mathcal{U})\Phi\;dx=d_s\langle g,\Phi	\rangle_{\cV^\star,\cV}\quad \forall \Phi\in\mathring{H}^1_L(y^\alpha,\mathcal{C})
\end{equation}
with $\alpha=1-2s$ and $d_s=2^\alpha \frac{\Gamma(1-s)}{\Gamma(s)}$, where we recall that $0<s<1$. That is, the function $\mathcal U\in \mathring{H}^1_L(y^\alpha,\mathcal{C})$ is a weak solution of the following  problem
\begin{equation}\label{edp}
\begin{cases}
\mbox{div}(y^\alpha\nabla \mathcal U)=0\;\;&\mbox{ in}\;\mathcal C\\
\frac{\partial\mathcal U}{\partial\nu^\alpha}+d_sf(x,\mathcal U)=d_s g\;\;\;&\mbox{ on }\;\Omega\times\{0\},
\end{cases}
\end{equation}
where we have set
\begin{align*}
\frac{\partial\mathcal U}{\partial\nu^\alpha}(x,0)=\lim_{y\to 0}y^\alpha\mathcal U_y(x,y)=\lim_{y\to 0}y^\alpha\frac{\partial\mathcal U(x,y)}{\partial y}.
\end{align*}

We have the following result.

\begin{lemma}\label{lem:CS}
Let Assumption \ref{assum3} on $f$  be fulfilled and $g\in \mathcal{V}^\star$ with $\mathcal{V}:=\mathbb H^{s}(\Omega)\cap L_F(\Omega)$ as defined at the beginning of Section \ref{sec:weaksol}. Then there exists a unique weak solution $\mathcal{U}\in \mathcal V_L:= \left\{v\in \mathring{H}^1_L(y^\alpha,\mathcal{C}): v|_{\Omega\times\{0\}}\in \mathcal{V} \right\}$ of \eqref{eq:extendedweak}. Furthermore, there holds $\mathcal{U}(\cdot,0)=u\in \mathcal{V}$, where $u$ represents the weak solution of \eqref{ellip-pro} according to \eqref{form-ws}.
\end{lemma}

\begin{proof}
We already know that if the solution $\mathcal{U}\in\mathring{H}^1_L(y^\alpha,\mathcal{C})$ of \eqref{eq:extendedweak} exists then $\mathcal U(\cdot,0)=u\in\mathbb H^s(\Omega)$. This is a trivial consequence of the corresponding result for linear problems. Therefore, we just have to prove the existence and uniqueness part.
Let us set
\begin{align*}
\mathcal E(\mathcal U,\Phi):=\int_\mathcal{C} y^\alpha \nabla \mathcal{U}\cdot\nabla \Phi\;dxdy+d_s\int_{\Omega}f(x,\mathcal{U})\,\Phi\;dx,\;\;\;\mathcal U,\Phi\in\mathcal V_L.
\end{align*}
Next, let $\mathcal U\in \mathcal{V}_L$ be fixed. It is clear that $\mathcal E(\mathcal U,\cdot)$ is linear in the second variable. Proceeding exactly as in the proof of Proposition \ref{prop-exis}, we get that $\mathcal E(\mathcal U,\cdot)\in \mathcal V_L^\star$. In addition, we have that $\mathcal E$ is strictly monotone, hemi-continuous and coercive. This finishes the proof.
\end{proof}

In contrast to the nonlocal fractional Dirichlet problem \eqref{ellip-pro}, the extended problem \eqref{eq:extendedweak} (or equivalently \eqref{edp}) is localized such that a discretization by standard finite elements becomes feasible. However, a direct discretization is still challenging due to the semi-infinite computational domain. As remedy, one can employ the exponential decay of the solution $\mathcal{U}$ in certain norms as $y$ tends to infinity, see \cite{NOS}. In this regard, a truncation of the semi-infinite cylinder is reasonable. This leads to a problem posed on the truncated cylinder $\mathcal{C}_\mathpzc{Y}$:
Given $g\in \mathcal{V}^\star$, find 
\begin{align*}
\mathcal{U}_\mathpzc{Y}\in \mathcal V_{L,\mathpzc{Y}}= \left\{v\in \mathring{H}^1_L(y^\alpha,\mathcal{C}_\mathpzc{Y}): v|_{\Omega\times\{0\}}\in \mathcal{V} \right\}
\end{align*}
 such that 
\begin{equation}\label{eq:truncatedweak}
\int_{\mathcal{C}_\mathpzc{Y}} y^\alpha \nabla \mathcal{U}_\mathpzc{Y}\cdot\nabla \Phi\;dxdy+d_s\int_{\Omega}f(x,\mathcal{U}_\mathpzc{Y})\Phi\;dx=d_s\langle g,\Phi	\rangle_{\cV^\star,\cV}\quad \forall \Phi\in\mathring{H}^1_L(y^\alpha,\mathcal{C}_\mathpzc{Y}).
\end{equation}

In view of the discretization error estimates in the next section, we do not need to estimate the truncation error for the semi-linear problems. Instead, we will use the corresponding results for linear problems.

\section{Discretizing the problem and proof of error estimates}\label{s:disc}
The discretization of the linear problem is outlined in \cite{NOS}. In fact, the theory there will build the basis for the discussion of the semi-linear problems presented in the further course of this section. For the convenience of the reader we will collect the main ingredients from the linear case before we turn towards the treatment of the semi-linear problems. From here on, we assume that the underlying domain $\Omega$ is convex and polyhedral. We notice that such a domain has a Lipschitz continuous boundary, see e.g. \cite{Chen}.

Due to the singular behavior of the solution towards the boundary $\Omega$, anistropically refined meshes are preferable since these can be used to compensate the singular effects. In our context such meshes are defined as follows:
Let $\mathscr{T}_\Omega=\{K\}$ be a conforming and quasi-uniform triangulation of $\Omega$, where $K\in \mathbb{R}^N$ is an element that is isoparametrically equivalent either to the unit cube or to the unit simplex in $\mathbb{R}^N$. We assume $\# \mathscr{T}_\Omega \sim M^N$. Thus, the element size $h_{\mathscr{T}_\Omega}$ fulfills $h_{\mathscr{T}_\Omega}\sim M^{-1}$. The collection of all these meshes is denoted by $\mathbb{T}_\Omega$. Furthermore, let $\mathcal{I}_\mathpzc{Y}=\{I\}$ be a graded mesh of the interval $[0,\mathpzc{Y}]$ in the sense that $[0,\mathpzc{Y}]=\bigcup_{k=0}^{M-1}[y_k,y_{k+1}]$ with
\[
y_k=\left(\frac{k}{M}\right)^\gamma\mathpzc{Y},\quad k=0,\ldots,M,\quad \gamma>\frac{3}{1-\alpha}=\frac{3}{2s}>1.
\]
Now, the triangulations $\mathscr{T}_\mathpzc{Y}$ of the cylinder $\mathcal{C}_\mathpzc{Y}$ are constructed as tensor product triangulations by means of $\mathscr{T}_\Omega$ and $\mathcal{I}_\mathpzc{Y}$. The definitions of both imply $\# \mathscr{T}_\mathpzc{Y} \sim M^{N+1}$.
Finally, the collection of all those anisotropic meshes $\mathscr{T}_\mathpzc{Y}$ is denoted by 
$\mathbb{T}$.

Now, we define the finite element spaces posed on the previously introduced 
meshes. 
For every $\mathscr{T}_\mathpzc{Y}\in \mathbb{T}$ the finite element spaces 
$\mathbb{V}(\mathscr{T}_\mathpzc{Y})$ are now defined by
\[
\mathbb{V}(\mathscr{T}_\mathpzc{Y}):=\{\Phi\in C^0(\overline{ \mathcal{C}_\mathpzc{Y}}):\Phi|_{T}\in\mathcal{P}_1(K)\oplus\mathbb{P}_1(I)\ \forall \;T=K\times I\in \mathscr{T}_\mathpzc{Y},\ \Phi|_{\partial_L\mathcal{C}_\mathpzc{Y}}=0\}.
\]
In case that $K$ in the previous definition is a simplex then $\mathcal{P}_1(K)=\mathbb{P}_1(K)$, the set of polynomials of degree at most $1$. If $K$ is a cube then $\mathcal{P}_1(K)$ equals $\mathbb{Q}_1(K)$, the set of polynomials of degree at most 1 in 
each variable.

Throughout the remainder of the paper, without any mention, $0<s<1$, $\alpha=1-2s$ and $d_s=2^{\alpha}\frac{\Gamma(1-s)}{\Gamma(s)}$.

Using the just introduced notation, the finite element discretization of \eqref{eq:truncatedweak} is given by the function $\mathcal{U}_{\mathscr{T}_\mathpzc{Y}}\in\mathbb{V}(\mathscr{T}_\mathpzc{Y})$ which solves the variational identity
\begin{equation}\label{eq:truncateddiscrete}
\int_{\mathcal{C}_\mathpzc{Y}} y^\alpha \nabla \mathcal{U}_{\mathscr{T}_\mathpzc{Y}}\cdot\nabla \Phi\;dxdy+d_s\int_{\Omega}f(x,\mathcal{U}_{\mathscr{T}_\mathpzc{Y}})\Phi\;dx=d_s\langle g,\Phi	\rangle_{\cV^\star,\cV}\quad \forall \Phi\in\mathbb{V}(\mathscr{T}_\mathpzc{Y}).
\end{equation}

We have the following result.

\begin{lemma}
Let Assumption \ref{assum3} on $f$ be fulfilled and $g\in \mathcal{V}^\star$. Then there exists a unique solution $\mathcal{U}_{\mathscr{T}_\mathpzc{Y}}\in\mathbb{V}(\mathscr{T}_\mathpzc{Y})$ of \eqref{eq:truncateddiscrete}.
\end{lemma}

\begin{proof}
	The existence of a solution can be proven by means of Browder's fixed-point theorem employing the monotonicity of the nonlinearity $f$. The uniqueness is a consequence of the $\mathring{H}^1_L(y^\alpha,\mathcal{C}_\mathpzc{Y})$-coercivity of the bilinearform in \eqref{eq:truncateddiscrete} and the monotonicity of $f$. Indeed, let $\mathcal{U}_1$ and $\mathcal{U}_2$ be two different solutions of \eqref{eq:truncateddiscrete}. Then we infer that there exists a constant $c>0$ such that
	\begin{align*}
		\|\mathcal{U}_1- & \mathcal{U}_2\|^2_{H^1_L(y^\alpha,\mathcal{C}_\mathpzc{Y})}\leq \\ 
	 &	c\left(\int_{\mathcal{C}_\mathpzc{Y}} y^\alpha |\nabla (\mathcal{U}_1-\mathcal{U}_2)|^2\;dxdy+d_s\int_{\Omega}(f(x,\mathcal{U}_1)-f(\cdot,\mathcal{U}_2))(\mathcal{U}_1-\mathcal{U}_2)\;dx\right)=0.
	\end{align*}
	Hence, $\mathcal U_1=\mathcal U_2$ and the proof is finished.
\end{proof}

For the error analysis it will be useful to have the intermediate solution $\tilde{\mathcal{U}}_{\mathscr{T}_\mathpzc{Y}}\in\mathbb{V}(\mathscr{T}_\mathpzc{Y})$ which solves the variational identity

\begin{equation}\label{eq:truncateddiscreteintermediate}
\int_{\mathcal{C}_\mathpzc{Y}} y^\alpha \nabla \tilde{\mathcal{U}}_{\mathscr{T}_\mathpzc{Y}}\cdot\nabla \Phi\;dxdy=d_s\langle g-f(\cdot,u),\Phi	\rangle_{\cV^\star,\cV}\quad \forall \Phi\in\mathbb{V}(\mathscr{T}_\mathpzc{Y}),
\end{equation}
where $u$ denotes the weak solution of \eqref{ellip-pro}. Since $\tilde{\mathcal{U}}_{\mathscr{T}_\mathpzc{Y}}$ represents the solution of a linear problem, corresponding error estimates are directly applicable.

\begin{lemma}\label{lemma:linear}
Let Assumptions \ref{assum3} and \ref{assum-lip} on $f$ be fulfilled and $g\in\mathbb{H}^{1-s}(\Omega)\cap L^p(\Omega)$ with $p$ as in \eqref{cond-p}. Moreover, let $u$ be the solution of \eqref{ellip-pro} and $\tilde{\mathcal{U}}_{\mathscr{T}_\mathpzc{Y}}$  the solution of \eqref{eq:truncateddiscreteintermediate}. Then there is a constant $c>0$ such that 
\[
\|u-\tilde{\mathcal{U}}_{\mathscr{T}_\mathpzc{Y}}\|_{H^s(\Omega)}\leq c|\log(\# \mathscr{T}_\mathpzc{Y})|^s(\# \mathscr{T}_\mathpzc{Y})^{-1/(N+1)}
\]
and
\[
\|u-\tilde{\mathcal{U}}_{\mathscr{T}_\mathpzc{Y}}\|_{L^2(\Omega)}\leq c|\log(\# \mathscr{T}_\mathpzc{Y})|^{2s}(\# \mathscr{T}_\mathpzc{Y})^{-(1+s)/(N+1)}
\]
provided that $\mathpzc{Y}\sim \log(\# \mathscr{T}_\mathpzc{Y})$.
\end{lemma}

\begin{proof}
 This is a consequence of \cite[Theorem 5.4 and Remark 5.5]{NOS} and \cite[Proposition 4.7]{NOS3} once we know that $f(\cdot,u)\in\mathbb{H}^{1-s}(\Omega)$. Since $g\in \mathbb{H}^{1-s}(\Omega)\cap L^p(\Omega)$ with $p$ as in \eqref{cond-p}, it follows from Corollary \ref{ellip-regula} that the unique weak solution $u$ belongs to $\mathbb H^{1+s}(\Omega)\hookrightarrow \mathbb H^{1-s}(\Omega)$ and hence $f(\cdot,u(\cdot))\in\mathbb H^{1-s}(\Omega)$ according to Lemma \ref{lem-f-lip}. This finishes the proof.
\end{proof}

For later purposes, we need to show that $\tilde{\mathcal{U}}_{\mathscr{T}_\mathpzc{Y}}$ is uniformly bounded in $L^\infty(\Omega)$, since we only assume a local Lipschitz condition for the nonlinearity $f$.

\begin{lemma}\label{lemma:Linftybounded}
Let Assumptions \ref{assum3} and \ref{assum-lip} on $f$ be fulfilled and $g\in \mathbb H^{1-s}(\Omega)\cap L^p(\Omega)$ with $p$ as in \eqref{cond-p}. Furthermore, let $s>(N-2)/2$. Then the solution $\tilde{\mathcal{U}}_{\mathscr{T}_\mathpzc{Y}}$ of \eqref{eq:truncateddiscreteintermediate} is uniformly bounded in $L^\infty(\Omega)$.
\end{lemma}

\begin{proof}
We denote by $I_{\mathscr{T}_\Omega}u$ the (modified) Clement interpolant of $u$, which is well defined for $u\in\mathbb{H}^{s}(\Omega)$. Next, let $K_*\in \mathscr{T}_\Omega$ be the element where $|u-I_{\mathscr{T}_\Omega}u|$ admits its supremum. By means of an inverse inequality, we deduce
\begin{align}
	\|u-\tilde{\mathcal{U}}_{\mathscr{T}_\mathpzc{Y}}\|_{L^\infty(\Omega)}&=\|u-\tilde{\mathcal{U}}_{\mathscr{T}_\mathpzc{Y}}\|_{L^\infty(K_*)}\leq \|u-I_{\mathscr{T}_\Omega}u\|_{L^\infty(K_*)}+\|I_{\mathscr{T}_\Omega}u-\tilde{\mathcal{U}}_{\mathscr{T}_\mathpzc{Y}}\|_{L^\infty(K_*)}\notag\\
	&\leq c\left(\|u-I_{\mathscr{T}_\Omega}u\|_{L^\infty(K_*)}+h_{K_*}^{-N/2}\|u-\tilde{\mathcal{U}}_{\mathscr{T}_\mathpzc{Y}}\|_{L^2(K_*)}\right),\label{eq:linfty}
\end{align}
where $h_{K_*}$ denotes the diameter of $K$. The first term in \eqref{eq:linfty} is bounded due to Theorem \ref{theo-bound}. For the second one, we notice that $h_{K_*}\sim h_{\mathcal{T}_\Omega}\sim M^{-1}$ and $\# \mathscr{T}_\mathpzc{Y}\sim M^{N+1}$. Consequently, the assertion follows from Lemma \ref{lemma:linear}.
\end{proof}

\begin{lemma}\label{lemma:intermediate}
Let Assumptions \ref{assum3} and \ref{assum-lip} on $f$ be fulfilled, $g\in \mathbb H^{1-s}(\Omega)\cap L^p(\Omega)$ with $p$ as in \eqref{cond-p} and $s>(N-2)/2$. Furthermore, let $u$, $\mathcal{U}_{\mathscr{T}_\mathpzc{Y}}$ and $\tilde{\mathcal{U}}_{\mathscr{T}_\mathpzc{Y}}$ be the solutions of \eqref{ellip-pro}, \eqref{eq:truncateddiscrete} and \eqref{eq:truncateddiscreteintermediate}, respectively. Then there is a constant $c>0$ such that
\[
\|\mathcal{U}_{\mathscr{T}_\mathpzc{Y}}-\tilde{\mathcal{U}}_{\mathscr{T}_\mathpzc{Y}}\|_{L^2(\Omega)}\leq c\|u-\tilde{\mathcal{U}}_{\mathscr{T}_\mathpzc{Y}}\|_{L^2(\Omega)}.
\]
\end{lemma}

\begin{proof}
	Due to the $\mathring{H}^1_L(y^\alpha,\mathcal{C}_\mathpzc{Y})$-coercivity of the bilinear form in \eqref{eq:truncateddiscrete} and \eqref{eq:truncateddiscreteintermediate}, and the monotonicity of $f$, we obtain that there is a constant $c>0$ such that
	
	\begin{align*}
	c\|\mathcal{U}_{\mathscr{T}_\mathpzc{Y}}-\tilde{\mathcal{U}}_{\mathscr{T}_\mathpzc{Y}}\|^2_{H^1_L(y^\alpha,\mathcal{C}_\mathpzc{Y})}&\leq \int_{\mathcal{C}_\mathpzc{Y}} y^\alpha \nabla (\mathcal{U}_{\mathscr{T}_\mathpzc{Y}}-\tilde{\mathcal{U}}_{\mathscr{T}_\mathpzc{Y}})\cdot\nabla (\mathcal{U}_{\mathscr{T}_\mathpzc{Y}}-\tilde{\mathcal{U}}_{\mathscr{T}_\mathpzc{Y}})\\
	&=d_s\int_{\Omega}(f(\cdot,\tilde{\mathcal{U}}_{\mathscr{T}_\mathpzc{Y}})-f(\cdot,\mathcal{U}_{\mathscr{T}_\mathpzc{Y}}))(\mathcal{U}_{\mathscr{T}_\mathpzc{Y}}-\tilde{\mathcal{U}}_{\mathscr{T}_\mathpzc{Y}}) \\
 &\quad +d_s\int_{\Omega}(f(\cdot,u)-f(\cdot,\tilde{\mathcal{U}}_{\mathscr{T}_\mathpzc{Y}}))(\mathcal{U}_{\mathscr{T}_\mathpzc{Y}}-\tilde{\mathcal{U}}_{\mathscr{T}_\mathpzc{Y}})\\
	&\leq d_s\int_{\Omega}(f(\cdot,u)-f(\cdot,\tilde{\mathcal{U}}_{\mathscr{T}_\mathpzc{Y}}))(\mathcal{U}_{\mathscr{T}_\mathpzc{Y}}-\tilde{\mathcal{U}}_{\mathscr{T}_\mathpzc{Y}}).
	\end{align*}
	Next, observe that both $u$ and $\tilde{\mathcal{U}}_{\mathscr{T}_\mathpzc{Y}}$ are uniformly bounded in $L^\infty(\Omega)$ according to Theorem \ref{theo-bound} and Lemma \ref{lemma:Linftybounded}. Consequently, the Cauchy-Schwarz inequality and the Lipschitz-continuity of the nonlinearity yield
	\begin{align}
	\int_{\Omega}(f(\cdot,u)-f(\cdot,\tilde{\mathcal{U}}_{\mathscr{T}_\mathpzc{Y}}))(\mathcal{U}_{\mathscr{T}_\mathpzc{Y}}-\tilde{\mathcal{U}}_{\mathscr{T}_\mathpzc{Y}})&\leq c \|u-\tilde{\mathcal{U}}_{\mathscr{T}_\mathpzc{Y}}\|_{L^2(\Omega)}\|\mathcal{U}_{\mathscr{T}_\mathpzc{Y}}-\tilde{\mathcal{U}}_{\mathscr{T}_\mathpzc{Y}}\|_{L^2(\Omega)}.\label{eq:lipschitz}
	\end{align}
	 Finally, the assertion can be deduced by means of the foregoing inequalities and the trace theorem of \cite[Proposition 2.1]{CDDS:11}, i.e.,
	\[
		\|\mathcal{U}_{\mathscr{T}_\mathpzc{Y}}-\tilde{\mathcal{U}}_{\mathscr{T}_\mathpzc{Y}}\|_{L^2(\Omega)}\leq c \|\mathcal{U}_{\mathscr{T}_\mathpzc{Y}}-\tilde{\mathcal{U}}_{\mathscr{T}_\mathpzc{Y}}\|_{H^1_L(y^\alpha,\mathcal{C}_\mathpzc{Y})},
	\]
	and the proof is finished.
\end{proof}

As a direct consequence of Lemmas \ref{lemma:linear} and \ref{lemma:intermediate}, we obtain the main result of this section.

\begin{theorem}\label{theorem:semilinear}
Let Assumptions \ref{assum3} and \ref{assum-lip} on $f$ be fulfilled, $g\in\mathbb{H}^{1-s}(\Omega)\cap L^p(\Omega)$ with $p$ as in \eqref{cond-p} and let $s>(N-2)/2$.
Moreover, let $u$ be the solution of \eqref{ellip-pro} and $\mathcal{U}_{\mathscr{T}_\mathpzc{Y}}$ the solution of \eqref{eq:truncateddiscrete}. Then there is a constant $c>0$ such that
\[
\|u-\mathcal{U}_{\mathscr{T}_\mathpzc{Y}}\|_{\mathbb H^s(\Omega)}\leq c|\log(\# \mathscr{T}_\mathpzc{Y})|^s(\# \mathscr{T}_\mathpzc{Y})^{-1/(N+1)}
\]
and
\[
\|u-\mathcal{U}_{\mathscr{T}_\mathpzc{Y}}\|_{L^2(\Omega)}\leq c|\log(\# \mathscr{T}_\mathpzc{Y})|^{2s}(\# \mathscr{T}_\mathpzc{Y})^{-(1+s)/(N+1)}
\]
provided that $\mathpzc{Y}\sim \log(\# \mathscr{T}_\mathpzc{Y})$.
\end{theorem}

We finally illustrate the results of Theorem~\ref{theorem:semilinear} by a numerical example. Let $N = 2$, $\Omega = (0,1)^2$. Under this setting, the eigenvalues and eigenfunctions of $-\Delta_D$ are:
\[
 \lambda_{k,l} = \pi^2 (k^2 + l^2) , \quad \varphi_{k,l}(x_1,x_2) = \sin(k\pi x_1) \sin(l\pi x_2) \quad k, l \in \mathbb{N} .  
\] 
Let the exact solution to \eqref{ellip-pro} be  
\begin{equation}\label{eq:uexact} 
 u = \lambda_{2,2}^{-s} \sin(2\pi x_1) \sin(2\pi x_2)
\end{equation}
and nolinearity $f(\cdot,u) = u^3=|u|^2u$. Using \eqref{ellip-pro} we immediately arrive at the expression for datum $g$. 

We use Newton's method to solve the nonlinear problem. The asymptotic relation $\|u-\mathcal{U}_{\mathscr{T}_\mathpzc{Y}}\|_{\mathbb H^s(\Omega)} \approx (\# \mathscr{T}_\mathpzc{Y})^{-1/3}$ is shown in Figure~\ref{f:HsL2} (left) for different choices of $s = 0.2, 0.4, 0.6$, and $s = 0.8$. We observe a quasi-optimal decay rate which confirms the $\mathbb{H}^s$-estimate in Theorem~\ref{theorem:semilinear}. We also present the $L^2$-error estimates in Figure~\ref{f:HsL2} (right), which decays as $(\# \mathscr{T}_\mathpzc{Y})^{-2/3}$ which is better than our theoretical prediction in Theorem~\ref{theorem:semilinear}. Notice that under the current literature status, theoretically, we cannot expect a better rate than Theorem~\ref{theorem:semilinear}, as we have used the linear result from \cite[Proposition 4.7]{NOS3} to prove Lemma~\ref{lemma:linear}.

\begin{figure}[h!]
\includegraphics[width=0.48\textwidth]{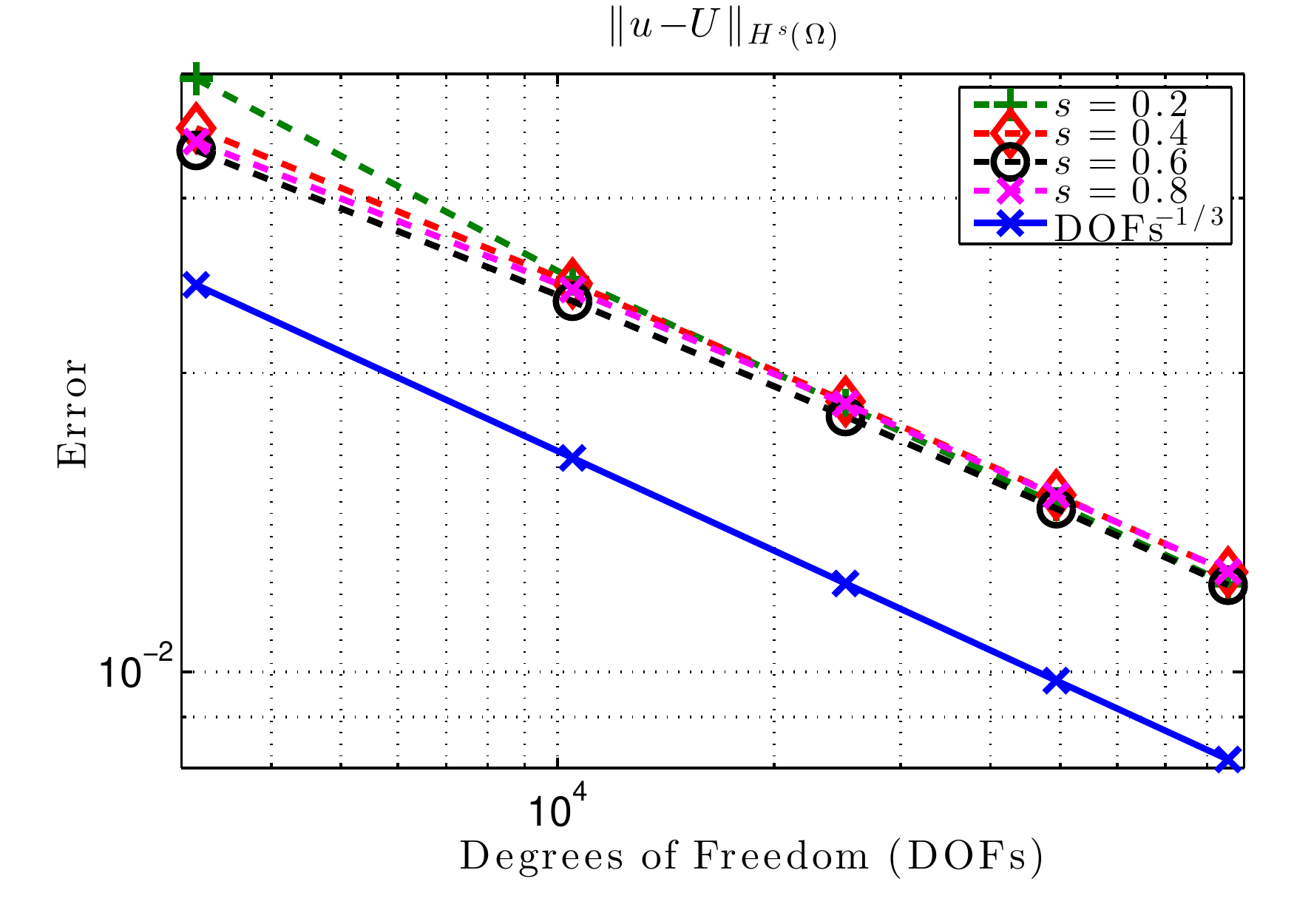} \quad 
\includegraphics[width=0.48\textwidth]{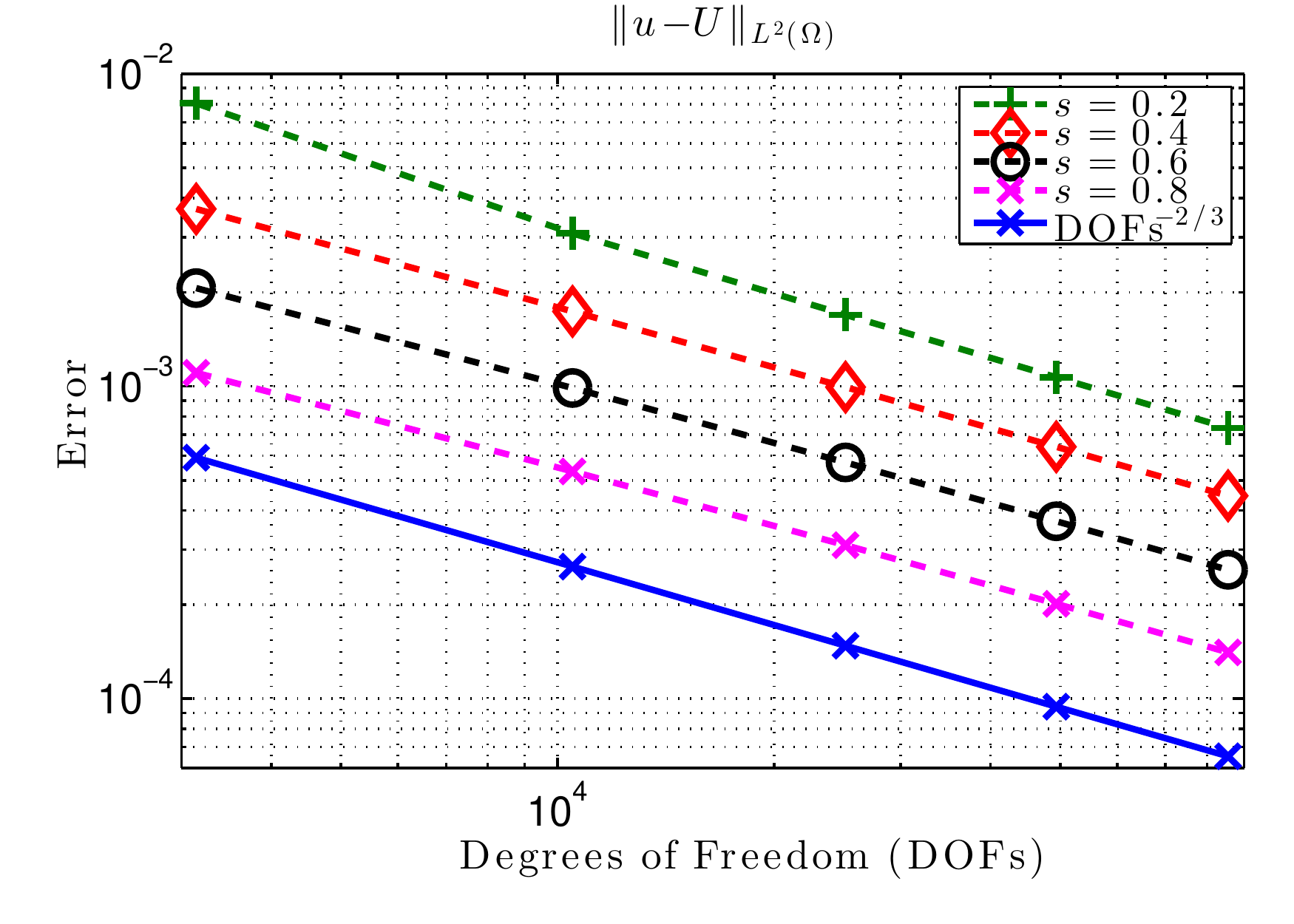} 
\caption{\label{f:HsL2} 
Rate of convergence on anisotropic meshes for $N = 2$ and $s = 0.2,0.4, 0.6$ and $s = 0.8$ is shown. $U$ is the numerical solution to \eqref{eq:truncateddiscrete} obtained by using Newton's method. On the other hand, $u$ is the exact solution given by \eqref{eq:uexact}. The blue line is the reference line. The left panel shows the $\mathbb{H}^s(\Omega)$-error, in all cases we recover $(\# \mathscr{T}_\mathpzc{Y})^{-1/3}$. The right panel shows the $L^2$-error which decays as $(\# \mathscr{T}_\mathpzc{Y})^{-2/3}$.}
\end{figure}

\bibliographystyle{plain}
\bibliography{lit}

\def\cprime{$'$}
\begin{thebibliography}{10}

\bibitem{NALD}
N.~Abatangelo and L.~Dupaigne.
\newblock Nonhomogeneous boundary conditions for the spectral fractional
  laplacian.
\newblock {\em Annales de l'Institut Henri Poincare (C) Non Linear Analysis},
  2016, to appear.

\bibitem{Abe2005403}
S.~Abe and S.~Thurner.
\newblock Anomalous diffusion in view of {E}instein's 1905 theory of {B}rownian
  motion.
\newblock {\em Physica A: Statistical Mechanics and its Applications},
  356(2--4):403 -- 407, 2005.

\bibitem{Adam}
R.A. Adams.
\newblock {\em Sobolev Spaces}.
\newblock Academic Press [A subsidiary of Harcourt Brace Jovanovich,
  Publishers], New York-London, 1975.
\newblock Pure and Applied Mathematics, Vol. 65.

\bibitem{atanackovic2014fractional}
T.M. Atanackovic, S.~Pilipovic, B.~Stankovic, and D.~Zorica.
\newblock {\em Fractional Calculus with Applications in Mechanics: Vibrations
  and Diffusion Processes}.
\newblock John Wiley \& Sons, 2014.

\bibitem{MR1736459}
E.~Barkai, R.~Metzler, and J.~Klafter.
\newblock From continuous time random walks to the fractional {F}okker-{P}lanck
  equation.
\newblock {\em Phys. Rev. E (3)}, 61(1):132--138, 2000.

\bibitem{Bie}
M.~Biegert and M.~Warma.
\newblock Some quasi-linear elliptic equations with inhomogeneous generalized
  {R}obin boundary conditions on ``bad'' domains.
\newblock {\em Adv. Differential Equations}, 15(9-10):893--924, 2010.

\bibitem{MR1081295}
J.-P. Bouchaud and A.~Georges.
\newblock Anomalous diffusion in disordered media: statistical mechanisms,
  models and physical applications.
\newblock {\em Phys. Rep.}, 195(4-5):127--293, 1990.

\bibitem{BV}
C.~Bucur and E.~Valdinoci.
\newblock Nonlocal diffusion and applications.
\newblock arXiv:1504.08292, 2015.

\bibitem{bio}
A.~Bueno-Orovio, D.~Kay, V.~Grau, B.~Rodriguez, and K.~Burrage.
\newblock Fractional diffusion models of cardiac electrical propagation: role
  of structural heterogeneity in dispersion of repolarization.
\newblock {\em J. R. Soc. Interface}, 11(97), 2014.

\bibitem{CT:10}
X.~Cabr{\'e} and J.~Tan.
\newblock Positive solutions of nonlinear problems involving the square root of
  the {L}aplacian.
\newblock {\em Adv. Math.}, 224(5):2052--2093, 2010.

\bibitem{Caf3}
L.~Caffarelli and L.~Silvestre.
\newblock An extension problem related to the fractional {L}aplacian.
\newblock {\em Comm. Part. Diff. Eqs.}, 32(7-9):1245--1260, 2007.

\bibitem{CaSt}
L.A. Caffarelli and P.R. Stinga.
\newblock Fractional elliptic equations, {C}accioppoli estimates and
  regularity.
\newblock {\em Ann. Inst. H. Poincar\'e Anal. Non Lin\'eaire}, 33(3):767--807,
  2016.

\bibitem{CDDS:11}
A.~Capella, J.~D{\'a}vila, L.~Dupaigne, and Y.~Sire.
\newblock Regularity of radial extremal solutions for some non-local semilinear
  equations.
\newblock {\em Comm. Part. Diff. Eqs.}, 36(8):1353--1384, 2011.

\bibitem{Chen}
L.~Chen.
\newblock Sobolev spaces and elliptic equations.
\newblock 2011.
\newblock \url{http://www.math.uci.edu/~chenlong/226/Ch1Space.pdf}.

\bibitem{wow}
W.~Chen.
\newblock A speculative study of $2/3$-order fractional laplacian modeling of
  turbulence: Some thoughts and conjectures.
\newblock {\em Chaos}, 16(2):1--11, 2006.

\bibitem{MR2035411}
L.~Debnath.
\newblock Fractional integral and fractional differential equations in fluid
  mechanics.
\newblock {\em Fract. Calc. Appl. Anal.}, 6(2):119--155, 2003.

\bibitem{MR2025566}
L.~Debnath.
\newblock Recent applications of fractional calculus to science and
  engineering.
\newblock {\em Int. J. Math. Math. Sci.}, (54):3413--3442, 2003.

\bibitem{NEGRETE}
D.~del Castillo-Negrete, B.~A. Carreras, and V.~E. Lynch.
\newblock Fractional diffusion in plasma turbulence.
\newblock {\em Physics of Plasmas}, 11(8):3854--3864, 2004.

\bibitem{NPV}
E.~Di~Nezza, G.~Palatucci, and E.~Valdinoci.
\newblock Hitchhikerʼs guide to the fractional {S}obolev spaces.
\newblock {\em Bulletin des Sciences Math{\'e}matiques}, 136(5):521--573, 2012.

\bibitem{Doman}
M.~Doman.
\newblock Weak uniform rotundity of {O}rlicz sequence spaces.
\newblock {\em Math. Nachr.}, 162:145--151, 1993.

\bibitem{GH:14}
P.~Gatto and J.~S. Hesthaven.
\newblock Numerical approximation of the fractional laplacian via hp-finite
  elements, with an application to image denoising.
\newblock {\em J. Sci. Comp.}, 65(1):249--270, 2015.

\bibitem{GU}
V.~Gol{\cprime}dshtein and A.~Ukhlov.
\newblock Weighted {S}obolev spaces and embedding theorems.
\newblock {\em Trans. Amer. Math. Soc.}, 361(7):3829--3850, 2009.

\bibitem{MR1926470}
R.~Gorenflo, F.~Mainardi, D.~Moretti, and P.~Paradisi.
\newblock Time fractional diffusion: a discrete random walk approach.
\newblock {\em Nonlinear Dynam.}, 29(1-4):129--143, 2002.
\newblock Fractional order calculus and its applications.

\bibitem{Gru2}
G.~Grubb.
\newblock Regularity of spectral fractional dirichlet and neumann problems.
\newblock {\em Math. Nachr.}, 2015.

\bibitem{ICH}
R.~Ishizuka, S.-H. Chong, and F.~Hirata.
\newblock An integral equation theory for inhomogeneous molecular fluids: The
  reference interaction site model approach.
\newblock {\em J. Chem. Phys}, 128(3), 2008.

\bibitem{kinderlehrer1980}
D.~Kinderlehrer and G.~Stampacchia.
\newblock {\em An Introduction to Variational Inequalities and their
  Applications}.
\newblock Academic Press, New York, 1980.

\bibitem{KO84}
A.~Kufner and B.~Opic.
\newblock How to define reasonably weighted {S}obolev spaces.
\newblock {\em Comment. Math. Univ. Carolin.}, 25(3):537--554, 1984.

\bibitem{MR2064019}
S.~Z. Levendorski{\u\i}.
\newblock Pricing of the {A}merican put under {L}\'evy processes.
\newblock {\em Int. J. Theor. Appl. Finance}, 7(3):303--335, 2004.

\bibitem{LM}
J.-L. Lions and E.~Magenes.
\newblock {\em Non-homogeneous boundary value problems and applications. {V}ol.
  {I}}.
\newblock Springer-Verlag, New York-Heidelberg, 1972.
\newblock Translated from the French by P. Kenneth, Die Grundlehren der
  mathematischen Wissenschaften, Band 181.

\bibitem{PSSB:PSSB2221330150}
R.R. Nigmatullin.
\newblock The realization of the generalized transfer equation in a medium with
  fractal geometry.
\newblock {\em Physica Status Solidi (b)}, 133(1):425--430, 1986.

\bibitem{NOS}
R.H. Nochetto, E.~Ot\'arola, and A.~J. Salgado.
\newblock A {PDE} approach to fractional diffusion in general domains: A priori
  error analysis.
\newblock {\em Found. Comput. Math.}, 15(3):733--791, 2015.

\bibitem{NOS3}
R.H. Nochetto, E.~Ot\'arola, and A.J. Salgado.
\newblock A {PDE} approach to space-time fractional parabolic problems.
\newblock arXiv:1404.0068, 2014.

\bibitem{MR1604710}
A.I. Saichev and G.M. Zaslavsky.
\newblock Fractional kinetic equations: solutions and applications.
\newblock {\em Chaos}, 7(4):753--764, 1997.

\bibitem{ST:10}
P.R. Stinga and J.L. Torrea.
\newblock Extension problem and {H}arnack's inequality for some fractional
  operators.
\newblock {\em Comm. Part. Diff. Eqs.}, 35(11):2092--2122, 2010.

\bibitem{Turesson}
B.O. Turesson.
\newblock {\em Nonlinear potential theory and weighted {S}obolev spaces},
  volume 1736 of {\em Lecture Notes in Mathematics}.
\newblock Springer-Verlag, Berlin, 2000.

\bibitem{War}
M.~Warma.
\newblock The fractional relative capacity and the fractional {L}aplacian with
  {N}eumann and {R}obin boundary conditions on open sets.
\newblock {\em Potential Anal.}, 42(2):499--547, 2015.

\end{thebibliography}

\end{document}